\documentclass[11pt]{amsart}
\usepackage{url}
\usepackage[a4paper,includeheadfoot,left=25mm,right=25mm,top=00mm,bottom=20mm,headheight=20mm]{geometry}
\usepackage{amssymb,amsmath,amsthm}
\usepackage{xcolor,graphicx}
\usepackage{verbatim}
\usepackage{hyperref}
\usepackage[utf8]{inputenc}
\usepackage{fancyhdr}

\newcommand{\sgn}{\operatorname{sgn}}
\newtheorem{thm}{Theorem}
\newtheorem{lem}[thm]{Lemma}

\newtheorem{prop}[thm]{Proposition}
\newtheorem{cor}[thm]{Corollary}

\theoremstyle{definition}
\newtheorem{defn}[thm]{Definition}

\theoremstyle{remark}
\newtheorem{rmk}[thm]{Remark}

\newcommand{\NN}{\mathbb N}              
\newcommand{\ZZ}{\mathbb Z}              
\newcommand{\RR}{\mathbb R}              
 
\newcommand{\D}{\ensuremath{\,\mathrm{d}}}

\renewcommand{\epsilon}{\varepsilon}
\renewcommand{\geq}{\geqslant}
\renewcommand{\leq}{\leqslant}

\newcommand{\FT}{\mathcal F}

\newcommand{\abs}[1]{\left\lvert#1\right\rvert}
\newcommand{\jp}[1]{\left\langle#1\right\rangle}
\usepackage{scalerel}
\usepackage{stackengine}
\stackMath
\newcommand\reallywidecheck[1]{%
\savestack{\tmpbox}{\stretchto{%
  \scaleto{%
    \scalerel*[\widthof{\ensuremath{#1}}]{\kern-.6pt\bigwedge\kern-.6pt}%
    {\rule[-\textheight/2]{1ex}{\textheight}}
  }{\textheight}%
}{0.5ex}}%
\stackon[1pt]{#1}{\scalebox{-1}{\tmpbox}}%
}

\newcommand\reallywidehat[1]{%
\savestack{\tmpbox}{\stretchto{%
  \scaleto{%
    \scalerel*[\widthof{\ensuremath{#1}}]{\kern-.6pt\bigwedge\kern-.6pt}%
    {\rule[-\textheight/2]{1ex}{\textheight}}
  }{\textheight}%
}{0.5ex}}%
\stackon[1pt]{#1}{\tmpbox}%
}

\newcommand\reallywidetilde[1]{%
\savestack{\tmpbox}{\stretchto{%
  \scaleto{%
    \scalerel*[\widthof{\ensuremath{#1}}]{\kern-.6pt\sim\kern-.6pt}%
    {\rule[-\textheight/2]{1ex}{\textheight}}
  }{\textheight}%
}{0.5ex}}%
\stackon[1pt]{#1}{\tmpbox}%
}

\title[Spatially quasiperiodic Benjamin-Ono]{Local well-posedness of the Benjamin-Ono equation with spatially quasiperiodic data}

\author{Sultan Aitzhan}
\address{Department of Mathematics, Drexel University, Philadelphia, 19104, USA}
\email{sa3697@drexel.edu}

\author{David M. Ambrose}
\address{Department of Mathematics, Drexel University, Philadelphia, 19104, USA}
\email{dma68@drexel.edu}

\begin{document}

\begin{abstract}
We consider the Benjamin-Ono equation in the spatially quasiperiodic setting.
We establish local well-posedness of the initial value problem with initial data in quasiperiodic Sobolev spaces.
This requires developing some of the fundamental properties of Sobolev spaces and the energy method for quasiperiodic functions.
We discuss prospects for global existence.  We demonstrate that while conservation laws still hold, these quantities no longer control the 
associated Sobolev norms, thereby preventing the establishment of global results by usual arguments.
\end{abstract}



\maketitle
\tableofcontents

\section{Introduction}
The Benjamin-Ono equation is a dispersive equation used in modeling water waves, given by
\[
u_t = uu_x + Hu_{xx}.
\]
Here, $H$ stands for the Hilbert transform, which is defined by 
\[
Hu(x) = \frac{1}{\pi} \mathrm{p.v.}\int^\infty_{-\infty} \frac{u(y)}{x - y} \D y
\]
on the real line and
\[
Hu(x) = \frac{1}{2\pi} \mathrm{p.v.}\int^{\pi}_{-\pi} u(y) \cot\Big(\frac{x-y}{2}\Big) \D y
\]
on the circle.
Amongst its features, the Benjamin-Ono equation constitutes an integrable system and possesses multi-soliton solutions (see \cite{Case}).
The well-posedness theory for periodic and real-line decaying data is rich, and for brevity we refer the reader to the results of \cite{KillipLaurensVisan, GerardKappelerTopalov} and references therein.
However, the case of quasiperiodic initial data, such as
\[
u_0(x) = \cos(x) + \cos(\sqrt{2} x),
\]
is yet to be fully addressed, and in this paper we undertake the case of sufficiently regular quasiperiodic initial data.

Amongst other dispersive equations, some significant progress for the quasiperiodic problem is made in the case of the KdV and the NLS equations. 
Namely, quasiperiodic KdV equation is addressed in \cite{TsugawaKdV, DamanikGoldstein,BinderDamanikGoldsteinLukic} and quasiperiodic NLS equations are considered in \cite{Schippa,DamanikYongFeiPartOne,DamanikYongFeiPartTwo,Fei}.
The case of quasiperiodic data for other dispersive equations, including more general ones, is also considered in \cite{JasonZhao,HagenPapenburg,DamanikYongFeiKdV,DamanikYongFeigBBM}.
Beyond dispersive PDE, we very briefly mention that there are quasiperiodic well-posedness results for Euler equations \cite{SunTopalov}, and computational results for quasiperiodic water waves in \cite{WilkeningZhaoJFM2021,WilkeningZhaoJNS2021,WilkeningZhaoJCP,WilkeningZhaoPA2023,DyachenkoSemenova}.

In what follows, we describe some of the above results in relation to our work.
Some of these results \cite{HagenPapenburg,JasonZhao,Schippa,TsugawaKdV} are local. 
The methods used are traditional, which include the Picard iteration style argument (as in the case of \cite{TsugawaKdV,HagenPapenburg,Schippa}), and energy estimates (as in \cite{JasonZhao}).
Our work also involves proving energy estimates; as a result, our paper is closer in spirit to \cite{JasonZhao}.
The key difference between our work and \cite{JasonZhao} is the following. 
Quasiperiodic functions are a subset of almost periodic functions, and \cite{JasonZhao} provides local well-posedness for almost periodic data.
However, the energy method as employed in \cite{JasonZhao}, which uses a localized Sobolev norm, cannot tell if the standard Sobolev regularity is preserved.

While the papers mentioned in the previous paragraph (including the present work) do provide solutions, these solutions are local in time.
In the standard case of the periodic and decaying data, these solutions become global by means of conservation laws.
This happens due to the fact that these conservation laws control the associated Sobolev norms.
In the quasiperiodic case, however, it does not seem obvious if there are any quantities that control the norms used in \cite{HagenPapenburg,JasonZhao,Schippa,TsugawaKdV}.
The lack of such quantities makes it extremely difficult to extend solutions to arbitary time of existence.

In the case of results in \cite{DamanikGoldstein,BinderDamanikGoldsteinLukic,DamanikYongFeiKdV,DamanikYongFeiPartOne,DamanikYongFeiPartTwo, DamanikYongFeigBBM}, these papers provide global existence and uniqueness to KdV and NLS-type equations. 
These works rely on studying spectral properties of related Schr\"{o}dinger operators, and thus are somewhat different from traditional approaches.
Furthermore, these works generally assume a Diophantine condition and consider Fourier series with exponential decay as the initial data.
Since exponential decay in the Fourier space implies that the data is analytic, these papers focus on a very specific class of data. 
We remark that the result of \cite{HagenPapenburg} is a local existence result, while also considering an analytic class of initial data, although this
work does not include any Diophantine assumptions. 
Furthermore, we note here that in a recent work \cite{DamanikYongFeigBBM}, a family of generalized Benjamin-Bona-Mahony (BBM) equations are studied.
In the case of BBM, they produce local well-posedness for initial data with a polynomial rate of decay; the authors there again use the 
combinatorial analysis method of their previous works.  As the present work uses a quasiperiodic Sobolev space, there is no assumption of exponential
decay of the Fourier series, and we likewise use no Diophantine conditions.

Unlike in the works listed in the above paragraph, we do not attempt to minimize the rate of decay or attempt to prove global existence. 
Instead, we would like to tred somewhere in-between and provide a framework that can be adapted to other dispersive equations, similarly to the results of \cite{HagenPapenburg,JasonZhao,Schippa}.
In particular, we use a regularized system which, when combined with the method of initial data regularization Bona and Smith \cite{BonaSmith}, permits to reproduce the classical results of \cite{SautTemam,AbdelouhabBonaFellandSaut} in the quasiperiodic setting.
While the proof is tailored for the Benjamin-Ono equation, the method is general enough to encompass more equations, including KdV equation. 

We choose the Benjamin-Ono equation specifically because the contraction mapping approach, 
as applied directly to the Benjamin-Ono equation, is known to fail at any regularity. 
In fact, the flow map from the initial data to solutions fails to be uniformly continuous \cite{KenigKoenig2003,Molinet2007}
and the bilinear estimates cannot hold at all, as indicated in \cite[Theorem 4.3.1]{HerrSebastianDissertation}.
This suggests that to have any theory of well-posedness, one at least needs to develop the energy method for the quasiperiodic Benjamin-Ono equation with sufficiently differentiable data.
In fact, in proving well-posedness for $L^2$-based data on torus \cite{Molinet2007,Molinet2008,MolinetPilod2012}, it was first shown that the estimates were satisfied on smooth functions.
Of course, this presupposes existence of smooth solutions, and the goal of this paper is to prove existence of their quasiperiodic counterparts.

We now describe quasiperiodic functions. 
For a fixed natural number $N > 1$, let $\alpha = (\alpha_1, \ldots, \alpha_N)$ satisfy $\alpha \cdot k\neq 0$ for $k \in \dot{\ZZ}^N = \ZZ^N \setminus \{ 0,\ldots, 0\}.$  
For $j = 1,\ldots,N$ let $f_j \in H^s(\RR\setminus 2 \pi \alpha_j^{-1}\ZZ)$ and consider functions of the form 
\begin{equation}\label{FourierSeriesExpansion1}
	f(x) = \sum_{j=1}^N f_j (x) = \sum_{j=1}^N \sum_{k \in \ZZ} \hat{f_j}(k)\exp(i \alpha_j kx),
\end{equation}
where $\hat{f_j}(k)$ are the Fourier coefficients.
The condition $\alpha \cdot k \neq 0$ ensures that $f(x)$ is not a periodic function. 
Instead, it is quasiperiodic. 
Note that the space of functions of the form \eqref{FourierSeriesExpansion1} does not form an algebra, so this space does not suit the analysis due to the presence of the quadratic term in Benjamin-Ono equation.
As such, we introduce the following space:
\begin{defn}[Quasiperiodic Sobolev spaces]
Let $s \in \RR,$ and define a space $H^{s}_{qp}$ of functions of the form
\[
u(x) = \sum_{k \in \ZZ^N} \hat{u}(k) \exp(i \alpha \cdot k x)
\]
with the following norm:
\[ 
\| u\|_{H^s_{qp}}^2 := \sum_{k \in \ZZ^N} (1 + k_1^2 + \ldots + k_N^2)^s \abs{\hat{u}}^2(k) = \sum_{k \in \ZZ^N} (1 + |k|^2)^s \abs{\hat{u}}^2(k) = \sum_{k \in \ZZ^N} \jp{|k|}^{2s} \abs{\hat{u}}^2(k).
\]
\end{defn}
Functions of the form \eqref{FourierSeriesExpansion1} clearly belong to $H^s_{qp}$ for large enough $s$.
The subscript $qp$ in $H^s_{qp}$ stands for quasiperiodic, and it is clear how these spaces generalize the usual Sobolev spaces for periodic functions.
We note that this space is already mentioned in \cite{Schippa} for NLS and KdV equations and \cite{SunTopalov} for Euler equations.

We state the main result of this paper.

\begin{thm}
Let $s > N/2+1$ and let $u_0 \in H^s_{qp}.$ 
There exists $T := T(s, \| u_0 \|_{H^s_{qp}})>0$ such that the Benjamin-Ono equation has a unique solution $u\in C([0,T];H^s_{qp})$.
Furthermore, the solution $u$ depends continuously on the initial data.
\end{thm}

Even though in our analysis we routinely use that $N \geq 2$, when $N=1$ we recover the classical results of \cite{AbdelouhabBonaFellandSaut,Iorio}.
The method of the proof is by regularizing the original Benjamin-Ono equation by truncation in the Fourier space:
\begin{equation}\label{LWPeq1}
	(u)_t = \chi_n[(\chi_n u) (\chi_n u_{x})] + \chi_n[H u_{xx}],
\end{equation}
where operators $\chi_n$ and $H$ are defined via Fourier transform:
\begin{align*}
	\FT\chi_n[u](k) &= \mathbb{I}_{|\alpha \cdot k|<n} \FT u(k) \\
	\FT H[u](k) &= - i \sgn(\alpha \cdot k) \FT u(k),
\end{align*}
where $\mathbb{I}_\Omega(x)$ is the usual indicator function on a set $\Omega$.
The definition of the quasiperiodic Hilbert transform is not new and has already appeared in a number of works, see \cite{WilkeningZhaoJFM2021,WilkeningZhaoJNS2021}.

\textit{Outline of the paper:}
In Section 2, we explain some properties of standard Sobolev spaces to $H^s_{qp}$ spaces.
We also elaborate on convergence and divergence conditions for the quasiperiodic Fourier series in $H^s_{qp}$ spaces.
In Section 3, we prove existence of solutions to the regularized system, obtain uniform estimates on the approximate solutions, as well as a Cauchy property of these solutions.
We also prove the properties of data regularization in the method of Bona and Smith. 
Finally, in Section 4, we prove local well-posedness of the original problem.
In Section 5, we demonstrate that the standard conservation laws still hold in the quasiperiodic case, and yet that these conservation laws do not control the norms.

\textit{Notation:}
\begin{itemize}
	\item For a given $k \in \ZZ^N$, write $|k|$ to denote the Euclidean norm of $k$ and let $\jp{k} = \sqrt{1 + |k|^2}$.
	\item Inequality $A \lesssim B$ means that $A \leq CB$ for some constant $C>0.$
	To indicate dependence of constant parameters, say on $\epsilon,$ we write $A \lesssim_{\epsilon} B$.
	\item Asymptotic notation $f(x) = o(g(x))$ for $x$ close to $x_0$ means that $f(x)/g(x) \to 0$ as $x \to x_0$. 
\end{itemize}

\section{Preliminary tools}
The following two lemmas allow us to make sense of quasiperiodic Fourier series.
\begin{prop}[Convergence of quasiperiodic Fourier series in $\mathcal{S}'$]
Let $s > (N-1)/2$. Let $\psi \in \mathcal{S}'$, $f \in H^s_{qp}$. Then, $\langle f, \psi \rangle < \infty$.
\end{prop}
\begin{proof}
Using the definition of Fourier series and applying Fubini theorem to intechange the integral and the sum, we obtain
\begin{align*}
	\langle f, \psi \rangle &= \int_{-\infty}^{\infty} f(x) \psi(x) \D x \\
	&= \sum_{k \in \dot{\ZZ}^N} \hat{f}(k) \int_{-\infty}^{\infty} \psi(x) \exp(i \alpha \cdot k x) \D x \\
	&= \sum_{k \in \dot{\ZZ}^N} \hat{f}(k) \hat{\psi}(\alpha \cdot k).
\end{align*}
Adding the product $\jp{k}^s \jp{k}^{-s}$ and applying the Cauchy-Schwarz inequality yields
\[ 
|\langle f, \psi \rangle| \leq \| f \|_{H^s_{qp}} \left[\sum_{k \in \dot{\ZZ}^N}\jp{k}^{-2s} |\hat{\psi}|^2(\alpha \cdot k) \right]^{1/2}.
\]
Using the inequality 
\[\jp{k}^{-2s} \leq \prod_{i=2}^{N} \jp{k_i}^{-2s/(N-1)}\]
yields
\[ 
	\sum_{k \in \dot{\ZZ}^N}\jp{k}^{-2s} |\hat{\psi}|^2(\alpha \cdot k) \leq \sum_{k \in \dot{\ZZ}^N} \prod_{i=2}^{N} \jp{k_i}^{-2s/(N-1)} |\hat{\psi}|^2(\alpha \cdot k).
\]
Rearrange the summation to obtain
\[
	\sum_{k \in \dot{\ZZ}^N} \prod_{i=2}^{N} \jp{k_i}^{-2s/(N-1)} |\hat{\psi}|^2(\alpha \cdot k) \leq \sum_{i=2}^{N-1} \sum_{k_i \in \ZZ} \prod_{i=2}^{N} \jp{k_i}^{-2s/(N-1)} \sum_{k_1 \in \ZZ} |\hat{\psi}|^2(\alpha \cdot k).
\]
and note that the sum in $k_1$ converges by integral test, for we have 
\[
	\int_{k_1 \in \ZZ} |\hat{\psi}|^2(\alpha \cdot k) \sim_{\alpha} \| \phi \|_{L^2_x}
\]
and $\phi \in L^2_x.$
Furthermore, this bound is independent of $k_2, \ldots, k_N$, so that we are left with estimating 
\[
	\sum_{i=2}^{N-1} \sum_{k_i \in \ZZ} \prod_{i=2}^{N} \jp{k_i}^{-2s/(N-1)}.
\]
Since $2s/(N-1) >1$, each of these sums converges and we obtain $\langle f, \psi \rangle < \infty$.
\end{proof}

The lemma belows shows that the restriction $s > (N-1)/2$ is necessary.
\begin{prop}[Divergence of quasiperiodic Fourier series in $\mathcal{S}'$]
Let $s \leq (N-1)/2$. Then, there exists a sequence $\{ f_n \} \subset H^s_{qp}$ which satisfies the following:
\begin{itemize}
	\item $\sup_n \| f_n \|_{H^s_{qp}} < \infty,$
	\item for some $\phi \in \mathcal{D}, \langle f_n, \mathcal{F}_\chi^{-1} \phi \rangle = \langle \mathcal{F}_x f_n, \phi \rangle \to \infty$ as $n \to \infty$. 
\end{itemize}
\end{prop}
\begin{proof}
We follow the lines of proof of \cite[Lemma 5.2]{TsugawaKdV}, but adapted to the $H^{s}_{qp}$ spaces, and supplying further details.
Without loss of generality, assume that the entries of $\alpha$ are all positive, i.e. $\alpha_i > 0$ for $i=1,\ldots, N$.
Let $\bar{k}$ denote a vector of the first $N-1$ entries of $k$, i.e. $\bar{k} = (k_1, \ldots, k_{N-1})$.

Let 
\begin{align*}
	A_n &= \{ x \in \RR^N: \jp{x} \leq n, 2|\alpha_N| \leq | \alpha \cdot x | \leq 4 |\alpha_N| \}, \\
	A &= \{ x \in \RR^N: 2|\alpha_N| \leq | \alpha \cdot x | \leq 4 |\alpha_N| \}, \\
	\hat{f_n}(k) &= \begin{cases} 
		\jp{k}^{1- N} (\log \jp{k})^{-1}, &\text{ if } k \in \dot{\ZZ}^N \cap A_n, \\
		0 &\text{ else }.
	\end{cases}
\end{align*}
Let $\phi(\xi) \in \mathcal{D}$ such that $\phi(\xi) = 1$ for $2 |\alpha_N| \leq | \xi | \leq 4 |\alpha_N|$ and $\phi(\xi) = 0$ for $k \leq |\alpha_N|$ or $k \geq 5|\alpha_N|.$
Since $\phi$ is in $\mathcal{D}$, it is also Schwartz. 
Since Fourier transform is an isometry on the Schwartz space, we have that $\mathcal{F}_\chi^{-1} \phi$ is also Schwartz.

Now, use the Fourier series and Fubini theorem to write 
\begin{align*}
	\langle f_n, \mathcal{F}_\chi^{-1} \phi \rangle &= \int_{-\infty}^{\infty} f_n(x) \mathcal{F}_\chi^{-1} \phi(x) \D x \\
	&= \sum_{k \in \cdot{\ZZ}^N} \hat{f_n}(k) \int_{-\infty}^{\infty} \mathcal{F}_\chi^{-1} \phi(x) \exp(i \alpha \cdot k x) \D x \\
	&= \sum_{k \in \cdot{\ZZ}^N} \hat{f_n}(k) \phi(\alpha \cdot k) \\
	&= \sum_{k \in \ZZ^N \cap A_n, |k| \geq 1} \jp{k}^{1- N} (\log \jp{k})^{-1}.
\end{align*}
As $n\to \infty,$ we need to show that the limit of the above sums diverges, i.e. that 
\[\sum_{k \in \ZZ^N, |k| \geq 1, k \in A} \jp{k}^{1- N} (\log \jp{k})^{-1} = \infty. \]
We prove divergence by integral test.
Indeed, consider the integral
\[\int_{x \in \RR^N, |x| \geq 1, x \in A} \jp{x}^{1- N} (\log \jp{x})^{-1} \D x. \]
Write $A$ as a disjoint union of $A_{+}$ and $A_{-}$, where 
\[
A_{\pm} = \{ x \in \RR^N: 2|\alpha_N| \leq \pm \alpha \cdot x \leq 4 |\alpha_N| \}.
\]
Then, it follows that 
\begin{align*}
	&\int_{x \in \RR^N, |x| \geq 1, x \in A} \jp{x}^{1- N} (\log \jp{x})^{-1} \D x \\
	&= \int_{x \in \RR^N, |x| \geq 1, x \in A_{+}} \jp{x}^{1- N} (\log \jp{x})^{-1} \D x + \int_{x \in \RR^N, |x| \geq 1, x \in A_{-}} \jp{x}^{1- N} (\log \jp{x})^{-1} \D x.
\end{align*}
We show that the integral over $A_{+}$ diverges; the argument for $A_{-}$ is almost the same.
Note that for $x\in A_{+}$, we have $2 \alpha_N \leq \bar{\alpha} \cdot \bar{x} + \alpha_N x_N \leq 4 \alpha_N.$
Subtracting $\bar{\alpha} \cdot \bar{x}$ and dividing both sides by $\alpha_N$ yields 
\[ 
2 - \frac{\bar{\alpha} \cdot \bar{x}}{\alpha_N} \leq x_N \leq 4 - \frac{\bar{\alpha} \cdot \bar{x}}{\alpha_N}.
\]
We then obtain that $\jp{x}^2 = 1+ |\bar{x}|^2 + x_N^2 \leq C^2 \jp{\bar{x}}^2,$ where $C$ is some positive constant with $C >2$.
Now, the integral over $A_{+}$ becomes
\begin{align*}
	\int_{x \in \RR^N, |x| \geq 1, x \in A_{+}} &\jp{x}^{1- N} (\log \jp{x})^{-1} \D x \\
	&= \int_{x_1} \ldots \int_{x_{N-1}} \int_{2 - \frac{\bar{\alpha} \cdot \bar{x}}{\alpha_N}}^{4 - \frac{\bar{\alpha} \cdot \bar{x}}{\alpha_N}}\jp{x}^{1- N} (\log \jp{x})^{-1} \D x_N \D x_{N-1} \ldots \D x_{1}.
\end{align*}
Since $\jp{x} \leq C \jp{\bar{x}}$, we bound the integrand from below by 
\begin{align*}
	\int_{x_1} &\ldots \int_{x_{N-1}} \int_{2 - \frac{\bar{\alpha} \cdot \bar{x}}{\alpha_N}}^{4 - \frac{\bar{\alpha} \cdot \bar{x}}{\alpha_N}}\jp{x}^{1- N} (\log \jp{x})^{-1} \D x_N \D x_{N-1} \ldots \D x_{1}\\
	&\geq \int_{x_1} \ldots \int_{x_{N-1}} \int_{2 - \frac{\bar{\alpha} \cdot \bar{x}}{\alpha_N}}^{4 - \frac{\bar{\alpha} \cdot \bar{x}}{\alpha_N}}\frac{1}{C\jp{\bar{x}}^{N-1} (\log C+ \log \jp{\bar{x}})} \D x_N \D x_{N-1} \ldots \D x_{1} \\
	&= \frac{2}{C}\int_{x_1} \ldots \int_{x_{N-1}} \frac{1}{\jp{\bar{x}}^{N-1} (\log C+ \log \jp{\bar{x}})} \D x_{N-1} \ldots \D x_{1},
\end{align*}
where we have used the fact that the integrand no longer depends on $x_N$.
For simplicity, relabel $x := \bar{x}$, so that we are left with showing that 
\[
	\int_{x \in \RR^{N-1}} \frac{1}{\jp{x}^{N-1} (\log C+ \log \jp{x})} \D x
\]
diverges.
This is straightforward. First, restrict the integration region to $\{ x \in\RR^{N-1}: |x| > C \}$: 
\[
	\int_{x \in \RR^{N-1}} \frac{1}{\jp{x}^{N-1} (\log C+ \log \jp{x})} \D x \geq \int_{x \in \RR^{N-1}: |x| > C} \frac{1}{\jp{x}^{N-1} (\log C+ \log \jp{x})} \D x.
\]
We can then bound $\log C+ \log \jp{x} \leq 2 \log \jp{x}$, so that
\[
	\int_{x \in \RR^{N-1}} \frac{1}{\jp{x}^{N-1} (\log C+ \log \jp{x})} \D x \geq \frac{1}{2}\int_{x \in \RR^{N-1}: |x| > C} \frac{1}{\jp{x}^{N-1} \log \jp{x}} \D x.
\]
Since $C>2$, we have 
\[ \jp{x}^{N-1} \log \jp{x} < 2|x|^{N-1} (\log 2 |x|) < 4 |x|^{N-1} \log |x|. \]
Therefore, we need to estimate 
\[
\int_{x \in \RR^{N-1}: |x| > C} \frac{1}{|x|^{N-1} \log |x|} \D x.
\]
Using polar coordinates yields that 
\[
\int_{x \in \RR^{N-1}: |x| > C} \frac{1}{|x|^{N-1} \log |x|} \D x \sim_{N-1} \int_{C}^\infty \frac{r^{N-2}}{r^{N-1} \log r} \D r = \int_{C}^\infty \frac{1}{r \log r} \D r. 
\]
The latter integral diverges, as can be shown by substitution $y = \log r$.

To show $\sup_n \| f_n \|_{H^s_{qp}} < \infty,$ note that
\begin{align*}
	\sup_n \| f_n \|_{H^s_{qp}}^2 &= \sup_n \sum_{k \in \ZZ^N, |k|\geq 1} \jp{k}^{2s} \abs{\hat{f_n}(k)}^2 \\
	&= \sup_n \sum_{k \in A_n} \jp{k}^{2s} \jp{k}^{2(1- N)} (\log \jp{k})^{-2} \\
	&\leq \sum_{\substack{ k \in \dot{\ZZ}^N  \\ 2|\alpha_N| \leq | \alpha \cdot k | \leq 4 |\alpha_N|}} \jp{k}^{2(s+1-N)} (\log \jp{k})^{-2}.
\end{align*}
We split the summation as follows:
\[
	\sum_{\substack{ k \in \dot{\ZZ}^N  \\ 2|\alpha_N| \leq | \alpha \cdot k | \leq 4 |\alpha_N|}} 
= 
\sum_{\substack{ k \in \dot{\ZZ}^N  \\ 2|\alpha_N| \leq | \alpha \cdot k | \leq 4 |\alpha_N| \\ |\bar{k}|\geq 1}} 
+
\sum_{\substack{ k \in \dot{\ZZ}^N  \\ 2|\alpha_N| \leq | \alpha \cdot k | \leq 4 |\alpha_N|\\ |\bar{k}| < 1}}.
\]
Note that the second sum over $|\bar{k}| < 1$ has only 6 summands. 
Since $k \in \dot{\ZZ}^N,$ we have $|k| \geq 1$.
But $|\bar{k}|< 1$ implies $k_1 = \ldots = k_{N-1} = 0,$ so that $|k_N| > 1$.
Now, the condition $2|\alpha_N| \leq | \alpha \cdot k | \leq 4 |\alpha_N|$ over this subset implies that $2 \leq |k_N| \leq 4$, and there are only 6 elements that satisfy this condition.
Hence, the sum over $k$'s such that $|\bar{k}| < 1$ is bounded, and we only need to show that the first sum converges.
For this, by integral test it is enough to show that the integral
\[
\int_{\substack{x \in \RR^N, |x| \geq 1 \\ x\in A \\ |\bar{x}| \geq 1}} \jp{x}^{2(s+1-N)} (\log \jp{x})^{-2} \D x
\]
converges.
Since we can write the region $A$ as disjoint union of $A_{+}$ and $A_{-}$, we only consider the integral of $A_{+}$. 
As before, over $A_{+}$ we know that 
\[ 
2 - \frac{\bar{\alpha} \cdot \bar{x}}{\alpha_N} \leq x_N \leq 4 - \frac{\bar{\alpha} \cdot \bar{x}}{\alpha_N},
\]
so that 
\begin{align*}
	\int_{\substack{x \in \RR^N \cap A_{+}, \\ |x| \geq 1, |\bar{x}| \geq 1}} &\jp{x}^{2(s+1-N)} (\log \jp{x})^{-2} \D x \\
	&= \int_{x_1} \ldots \int_{x_{N-1}} \int_{2 - \frac{\bar{\alpha} \cdot \bar{x}}{\alpha_N}}^{4 - \frac{\bar{\alpha} \cdot \bar{x}}{\alpha_N}}\jp{x}^{2(s+1-N)} (\log \jp{x})^{-2} \D x_N \D x_{N-1} \ldots \D x_{1}.
\end{align*}
Since $\jp{\bar{x}} \leq \jp{x}$, we can bound the integrand $\jp{x}^{2(s+1-N)} (\log \jp{x})^{-2}$ by $\jp{\bar{x}}^{2(s+1-N)} (\log \jp{\bar{x}})^{-2}.$
Therefore, 
\begin{align*}
	\int_{\substack{x \in \RR^N, |x| \geq 1 \\ x\in A_{+} \\ |\bar{x}| \geq 1}} &\jp{x}^{2(s+1-N)} (\log \jp{x})^{-2} \D x \\
	&\leq 2 \int_{x_1} \ldots \int_{x_{N-1}} \jp{\bar{x}}^{2(s+1-N)} (\log \jp{\bar{x}})^{-2} \D x_{N-1} \ldots \D x_{1}\\
	&= 2 \int_{x \in \RR^{N-1}, |x|\geq 1}\jp{x}^{2(s+1-N)} (\log \jp{x})^{-2} \D x,
\end{align*}
where we have relabeled $x := \bar{x}$ for convenience.
The last integral can be handled by polar coordinates: we have
\[
\int_{x \in \RR^{N-1}, |x|\geq 1} \jp{x}^{2(s+1-N)} (\log \jp{x})^{-2} \D x \sim \int_{1}^{\infty} r^{N-2}\jp{r}^{2(s+1-N)} (\log \jp{r})^{-2} \D r.
\]
Since $r < \jp{r}$, the latter is bounded above by 
\[
	\int_{1}^{\infty} r^{N-2}\jp{r}^{2(s+1-N)} (\log \jp{r})^{-2} \D r \leq \int_{1}^{\infty} r^{N-2 + 2(s+1-N)}  (\log r)^{-2} \D r =	\int_{1}^{\infty} r^{2s-N}  (\log r)^{-2} \D r.
\]
Since $s \leq (N-1)/2$, we have $2s \leq N-1,$ so that 
\[
	\int_{1}^{\infty} r^{2s-N}  (\log r)^{-2} \D r \leq \int_{1}^{\infty} r^{-1}  (\log r)^{-2} \D r.
\]
The last integral is bounded; this can be shown by substitution $y = \log r$.
\end{proof}

With more regularity, a quasiperiodic Fourier series makes sense pointwise.
\begin{prop}[Sobolev inequality, $L^\infty$ bound]\label{SobolevInequality}
Assume $s > N/2$. Let $u \in H^{s}_{qp}.$ Then, 
\[
\sum_{k \in \ZZ^N} |\hat{u}|(k) \lesssim \| u\|_{H^{s}_{qp}}.
\]
Furthermore, $u$ is bounded and continuous on $\RR$.
\end{prop}
\begin{proof}
The Sobolev inequality is proved by following the arguments of \cite[Lemma 6.5]{follandPDEs}.
\end{proof}

\begin{prop}[Algebra property]
Assume $s > N/2.$ Let $u, v \in H^{s}_{qp}.$ Then  
\[
\| uv\|_{H^{s}_{qp}} \lesssim_s \| u\|_{H^{s}_{qp}} \| v\|_{H^{s}_{qp}}.
\]
\end{prop}
\begin{proof}
The algebra property is obtained by following the arguments of \cite[Chapter 6, Exercise 4]{follandPDEs}. 
\end{proof}

\begin{prop}[Fractional Leibniz rule]\label{FractionalLeibniz}
Define an operator $D^s_x$ by $\FT{D^s_x u} (k)= |k|^s \hat{u}(k).$ 

Let $s > 1$ and $s_0 > N/2$.
Let $u \in H^s_{qp} \cap H^{s_0 + 1}_{qp}$ and $v \in H^{s_0}_{qp} \cap H^{s-1}_{qp}$.
Then,
\[
\| D^s_x (uv) - u D^s_x v \|_{L^2_{qp}} \lesssim \| u \|_{H^{s}_{qp}} \| v\|_{H^{s_0}_{qp}} + \| u \|_{H^{s_0+1}_{qp}} \| v\|_{H^{s-1}_{qp}}.
\]
\end{prop}
\begin{proof}
To begin, use the convolution theorem to write:
\begin{align*}
	\FT{(D^s_x (uv) - u D^s_x v)}(k) &= |k|^s \FT{uv}(k) - \FT{u D^s_x v}(k) \\
	&= |k|^s \sum_{j \in \ZZ^N} \hat{u}(j)\hat{v}(k-j) - \sum_{j \in \ZZ^N} \hat{u}(j) |k-j|^s \hat{v}(k-j) \\
	&= \sum_{j \in \ZZ^N} \big[ |k|^s - |k-j|^s \big]\hat{u}(j)  \hat{v}(k-j).
\end{align*}
Thus, 
\begin{align*}
	\| D^s_x (uv) - u D^s_x v \|_{L^2_{qp}}^2 &= \sum_{k \in \ZZ^N} \abs{ \sum_{j \in \ZZ^N} \big[ |k|^s - |k-j|^s \big]\hat{u}(j)  \hat{v}(k-j)}^2 \\
	&\leq \sum_{k \in \ZZ^N} \left| \sum_{j \in \ZZ^N} \abs{ |k|^s - |k-j|^s} |\hat{u}(j)|  |\hat{v}(k-j)|\right|^2.
\end{align*}
Using the following inequality, from \cite{SautTemam},
\[ 
	\abs{ |k|^s - |k-j|^s} \lesssim_{s} |j|^{s} + |k-j|^{s-1} |j|,
\]
as well as triangle inequality yields the following bound 
\begin{align*}
	&\| D^s_x (uv) - u D^s_x v \|_{L^2_{qp}}^2 \\
	&\qquad \lesssim \sum_{k \in \ZZ^N} \left| \sum_{j \in \ZZ^N}|j|^{s}|\hat{u}(j)|  |\hat{v}(k-j)| \right|^2 + \sum_{k\in\ZZ^N}\left| \sum_{j \in \ZZ^N}| |k-j|^{s-1} | \alpha \cdot j| |\hat{u}(j)|  |\hat{v}(k-j)|\right|^2.
\end{align*}
It remains to bound the two sums. 

One can bound the first sum by Young's convolution inequality:
\[
	\sum_{k \in \ZZ^N} \left| \sum_{j \in \ZZ^N}|j|^{s}|\hat{u}(j)|  |\hat{v}(k-j)|\right|^2 
	\lesssim \left( \sum_{k \in \ZZ^N} |j|^{2s}|\hat{u}(j)|^2 \right) \left(\sum_{k \in \ZZ^N} |\hat{v}(k)| \right)^2.
\]
For any $s_0 > N/2,$ Sobolev inequality (Proposition \ref{SobolevInequality}) yields
\[ 
\left(\sum_{k \in \ZZ^N} |\hat{v}(k)| \right)^2 \lesssim \| v \|_{H^{s_0}_{qp}}^2,
\]
so that the first sum is bounded by 
\[ 
\| D^s_x u \|_{L^2_{qp}}^2 \| v \|_{H^{s_0}_{qp}}^2 \lesssim \| u \|_{H^s_{qp}}^2 \| v \|_{H^{s_0}_{qp}}^2.
\]
For the second sum, we similarly apply Young's convolution inequality to obtain
\[
	\sum_{k\in\ZZ^N}\left|\sum_{j \in \ZZ^N} |k-j|^{s-1} | \alpha \cdot j| |\hat{u}(j)|  |\hat{v}(k-j)|\right|^2 \lesssim 
	\left( \sum_{k \in \ZZ^N} |\alpha \cdot k| |\hat{u}(k)| \right)^2 \left(\sum_{k \in \ZZ^N} \left| k\right|^{2(s-1)} |\hat{v}(k)|^2 \right).
\]
One then applies the Sobolev inequality (Proposition \ref{SobolevInequality}, leading to 
\[ 
	\sum_{k \in \ZZ^N} |\alpha \cdot k| |\hat{u}(k)| \lesssim \| u \|_{H^{s_0 + 1}_{qp}}.
\]
This lets us bound the second sum by 
\[ 
	\| u \|_{H^{s_0 + 1}_{qp}}^2 \| D^{s-1}_{x}v\|_{L^2_{qp}}^2 \lesssim \| u \|_{H^{s_0 + 1}_{qp}}^2 \| v \|_{H^{s-1}_{qp}}^2.
\]
Therefore, we obtain 
\[ 
	\| D^s_x (uv) - u D^s_x v \|_{L^2_{qp}}^2 \lesssim  \| u \|_{H^s_{qp}}^2 \| v \|_{H^{s_0}_{qp}}^2 + \| u \|_{H^{s_0 + 1}_{qp}}^2 \| v \|_{H^{s-1}_{qp}}^2,
\]
as desired.
\end{proof}

Another tool to be used in the analysis is the interpolation inequality.
\begin{lem}\label{InterpolationInequality}
Let $N \in \NN.$ Let $l > N/2$ and let $u \in H^{l}_{qp}.$ Then, for any $p$ with $p\leq l,$ we have
\[ 
\|u\|_{H^{p}_{qp}} \leq \|u\|_{H^{l}_{qp}}^{p/l} \|u\|_{H^{0}_{qp}}^{1-p/l}.
\]
\end{lem}
\begin{proof}
The proof is essentially the same as in the case of usual Sobolev spaces, so we omit it.
\end{proof}

Finally, since the quasiperiodic integrals are averaged integrals, we need to justify taking their time derivatives. 
The following two propositions provide this justification.

\begin{prop}\label{limitofQPintegrals}
Suppose $f_n \to f$ uniformly on $\RR$, and that the averaged integrals 
\[ 
\lim_{R\to \infty} \frac{1}{2R} \int^{R}_{-R} f \D x, \qquad \lim_{R\to \infty} \frac{1}{2R} \int^{R}_{-R} f_n \D x,
\]
exist, for $n=0,1,2,\ldots.$ Then, 
\[ 
	\lim_{R\to \infty} \frac{1}{2R} \int^{R}_{-R} f \D x = \lim_{n \to \infty} \lim_{R\to \infty} \frac{1}{2R} \int^{R}_{-R} f_n \D x.
\]
\end{prop}
\begin{proof}
We adapt the proof given in \cite[Theorem 7.16]{babyrudin}. Let 
\[ \epsilon_n = \sup_{x \in\RR} |f_n(x) - f(x)|. \]
Thus, $f_n - \epsilon_n \leq f \leq f_n + \epsilon_n$, and integrating from $-R$ to $R$ yields
\begin{align*}
	\frac{1}{2R} \int^R_{-R} f_n \D x - \epsilon_n &\leq \frac{1}{2R} \int^R_{-R} f \D x = \frac{1}{2R} \int^R_{-R} f_n \D x  + \epsilon_n.
\end{align*} 
Hence 
\[
- \epsilon_n \leq \frac{1}{2R} \int^R_{-R} f \D x - \frac{1}{2R} \int^R_{-R} f_n \D x \leq \epsilon_n.
\]
Taking the limit in $R$ yields
\[ 
	- \epsilon_n \leq \lim_{R\to\infty}\frac{1}{2R} \int^R_{-R} f \D x - \lim_{R\to\infty}\frac{1}{2R} \int^R_{-R} f_n \D x \leq \epsilon_n,
\]
so that 
\[ \abs{ \lim_{R\to\infty}\frac{1}{2R} \int^R_{-R} f \D x - \lim_{R\to\infty}\frac{1}{2R} \int^R_{-R} f_n \D x} \leq \epsilon_n.\]
Since $f_n \to f$ uniformly, we must have $\epsilon_n \to 0$ as $n \to \infty$. 
As the difference of averaged integrals is bounded by $\epsilon_n$, we see that taking $n\to \infty$ yields the desired result.
\end{proof}

With the preceding proposition in mind, we can prove the following result:
\begin{prop}\label{ExchangeDerivandInt}
Assume that the following averaged integrals exist for any $t\in(0,T)$
\[
\frac{\D}{\D t} \lim_{R \to \infty} \frac{1}{2R}\int^R_{-R} \phi(x,t) \D x,  \qquad
\lim_{R \to \infty} \frac{1}{2R}\int^R_{-R} \frac{\partial}{\partial t} \phi(x,t) \D x,
\]
and that $\dfrac{\partial}{\partial t} \phi(x,t) = \phi_t(x,t)$ is continuous in $t$. Then, 
\[\frac{\D}{\D t} \lim_{R \to \infty} \frac{1}{2R} \int^R_{-R} \phi(x,t) \D x=  \lim_{R \to \infty} \frac{1}{2R} \int^R_{-R} \frac{\partial}{\partial t} \phi(x,t) \D x, 
\]
i.e. we can interchange time derivatives and averaged integrals.
\end{prop}
\begin{proof}
We follow \cite[Theorem 9.42]{babyrudin}. 
Fix $t \in (0,T), x\in \RR$ and consider difference quotients
\[ \psi(x,s) = \frac{\phi(x,s) - \phi(x,t)}{s-t} \]
for $s$ sufficiently close to $t$. 
By the mean value theorem applied on $\phi$ in $s$, there exists an $r$ between $s$ and $t$ such that
\[
	\psi(x,s) = \frac{\partial}{\partial t} \phi(x,r).
\]
Therefore, by continuity of $\phi_t$, for any $\epsilon > 0$, there exists a $\delta > 0$ such that 
\[ \abs{ \psi(x,s) - \frac{\partial}{\partial t} \phi(x,s)} <\epsilon, \qquad (x \in \RR, 0 < |s-t| <\delta) \]
In other words, $\psi(x,s) \to \frac{\partial}{\partial t} \phi(x,t)$ uniformly, as $s \to t$.
Applying Proposition \ref{limitofQPintegrals} yields that
\[
	\lim_{R\to \infty} \frac{1}{2R} \int^{R}_{-R} \frac{\partial}{\partial t} \phi(x,t) \D x = \lim_{s \to t} \lim_{R\to \infty} \frac{1}{2R} \int^{R}_{-R} \psi(x,s)\D x.
\]
Now, let 
\[
f(y) = \lim_{R\to \infty} \frac{1}{2R} \int^{R}_{-R} \phi(x,y) \D x,
\] 
and note that 
\[
	\lim_{R\to \infty}\frac{1}{2R} \int^{R}_{-R} \psi(x,s) \D x = \frac{f(s) - f(t)}{s-t}
\]
by definition of $\psi(x,s).$
With this, we obtain that 
\[
	\lim_{R\to \infty} \frac{1}{2R} \int^{R}_{-R} \frac{\partial}{\partial t} \phi(x,t) \D x = \lim_{s \to t} \lim_{R\to \infty} \frac{1}{2R} \int^{R}_{-R} \psi(x,s)\D x = \lim_{s \to t} \frac{f(s) - f(t)}{s-t}.
\]
Since
\[\lim_{s \to t} \frac{f(s) - f(t)}{s-t} = \frac{\D}{\D t} f(t) = \frac{\D}{\D t} \lim_{R\to \infty} \frac{1}{2R} \int^{R}_{-R} \phi(x,t) \D x, \]
we clearly have 
\[
	\lim_{R\to \infty} \frac{1}{2R} \int^{R}_{-R} \frac{\partial}{\partial t} \phi(x,t) \D x =\frac{\D}{\D t} \lim_{R\to \infty} \frac{1}{2R} \int^{R}_{-R} \phi(x,t) \D x,
\]
as desired.
\end{proof}
To conclude, this proposition says that so long as the integrand is $C^1$ in time, we can safely exchange time derivatives.

Finally, we prove a proposition about Hilbert transform that we use repeatedly.
\begin{prop}\label{lineartermvanishes}
Let $s \in \RR.$ Then, for any $u \in H^s_{qp},$ 
\[ 
\sum_{k \in \ZZ^N} \FT{H \partial_x^2 D^s_x u}(k) \FT{D^s_x u}(-k) = 0.
\]
The statement is still true if $\partial^s_x u$ is replaced by $D^s_x u$.
\end{prop}
\begin{proof}
First, use the definition of Hilbert transform:
\begin{align*}
	\sum_{k \in \ZZ^N} &\FT{H \partial_x^2 D^s_x  u}(k) \FT{D^s_x  u}(-k) \\
	&=  \sum_{k \in \ZZ^N}(-i) \sgn(\alpha \cdot k) (i \alpha \cdot k)^2 \FT{D^s_x u}(k) \FT{D^s_x  u}(-k).
\end{align*}
Then, split the sum into two sums, depending on the sign of $\alpha \cdot k:$
\begin{align*}
	&\sum_{k: \alpha \cdot k> 0}(-i) \sgn(\alpha \cdot k) (i \alpha \cdot k)^2 \FT{D^s_x u}(k) \FT{D^s_x u}(-k) \\
	&\qquad \qquad + \sum_{k: \alpha \cdot k< 0}(-i) \sgn(\alpha \cdot k) (i \alpha \cdot k)^2 \FT{D^s_x  u}(k) \FT{D^s_x u}(-k).
\end{align*}
Evaluating the sign leads to
\begin{align*}
	\sum_{k: \alpha \cdot k> 0}(-i)(i \alpha \cdot k)^2 \FT{D^s_x u}(k) \FT{D^s_x u}(-k) - \sum_{k: \alpha \cdot k< 0}(-i) (i \alpha \cdot k)^2 \FT{D^s_x u}(k) \FT{D^s_x u}(-k).
\end{align*}
Performing a change of variables $k \mapsto -k$ in the second sum shows that the two sums are the same, thus the difference is $0$, as desired.
\end{proof}

\section{Regularized solutions and their properties}
In this section, we prove several properties of the regularized Benjamin-Ono equation \eqref{LWPeq1} with quasiperiodic initial data.
In particular, we show existence of approximate solutions, that these solutions are bounded uniformly in $n,$ and a Cauchy estimate for these solutions.

\subsection{Properties of regularized data}
In applying the Bona-Smith method \cite{BonaSmith}, regularization of initial data plays a crucial role.
As we seek to emulate their argument especially as to demonstrating full continuity of solutions with respect to the initial data, 
we also regularize the quasiperiodic data: for a given $u \in H^{s}_{qp},$ define $u_{\delta}$ through the equation
\begin{equation}\label{regData}
	\FT{(u_\delta)}(k) = \mathbb{I}_{|k|\leq \delta}(k) \FT{u}(k),
\end{equation}
where $|k|$ stands for the Euclidean norm of $k \in \ZZ^N$.

The following proposition outlines the advantages of regularizing data as we do. It is analagous to \cite[Lemma 5]{AMW}, and is essential for demonstrating
the full result of continuity with respect to the initial data.
\begin{prop}
Let $\mathcal{K} \subset H^{s}_{qp}$ be a compact set. Then, for any $\delta >1$ and any $u \in \mathcal{K}, (u)_\delta$ satisfies:
\begin{align}
	\label{RD1} \| (u)_\delta \|_{H^{s+j}_{qp}} &\lesssim_j \delta^j \| (u)_\delta \|_{H^{s}_{qp}} \quad \forall j \geq 0, \\
	\label{RD2} \| (u)_\delta - u \|_{L^2_{qp}} &= o(\delta^{-s}), \\
	\label{RD3} \| (u)_\delta - u \|_{H^s_{qp}} &= o(1).
\end{align}
In particular, the rate of convergence in \eqref{RD2} and \eqref{RD3} depends only on $\mathcal{K}$ and not on $u$.
Here, recall that the notation $f(x) = o(g(x))$ means that $f(x)/g(x) \to 0$ as $x \to \infty.$
\end{prop}

\begin{proof}
First, we prove \eqref{RD1}. Note that 
\begin{align*}
	\| (u)_\delta \|_{H^{s+j}_{qp}}^2 &= \sum_{k \in \ZZ^N} (1+ |k|^2)^s \abs{\hat{u}}^2(k) (1+ |k|^2)^j \mathbb{I}_{|k|\leq \delta}(k) \\
	&\leq (1+ \delta^2)^j \sum_{k \in \ZZ^N} (1+ |k|^2)^s \abs{\hat{u}}^2(k) \mathbb{I}_{|k|\leq \delta}(k) \\
	&\leq (2\delta^2)^j\sum_{k \in \ZZ^N} (1+ |k|^2)^s \abs{\hat{u}}^2(k),
\end{align*}
so that 
\[
	\| (u)_\delta \|_{H^{s+j}_{qp}} \lesssim \delta^{j} \| (u)_\delta \|_{H^{s}_{qp}}.
\]
As for \eqref{RD2}, we have 
\begin{align*}
	\| (u)_\delta - u \|_{L^2_{qp}}^2 &= \sum_{k \in \ZZ^N} \left(\mathbb{I}_{|k|\leq \delta}(k) - 1\right)^2 \abs{\hat{u}}^2(k) \\
	&= \sum_{k \in \ZZ^N: |k| > \delta} \abs{\hat{u}}^2(k) \\
	&\leq \frac{1}{\delta^{2s}}\sum_{k \in \ZZ^N: |k| > \delta} |k|^{2s} \abs{\hat{u}}^2(k).
\end{align*}
Clearly $\sum_{k \in \ZZ^N: |k| > \delta} |k|^{2s} \abs{\hat{u}}^2(k) \to 0$ as $\delta \to \infty$. 
Thus, we obtain \eqref{RD2}.

Finally, by the same means one obtains 
\begin{align*}
	\| (u)_\delta - u \|_{H^s_{qp}}^2 &= \sum_{k \in \ZZ^N: |k| > \delta} \jp{|k|}^{2s} \abs{\hat{u}}^2(k),
\end{align*}
which is the same as \eqref{RD3}. 

We tackle the last statement about the rate of convergence. 
For simplicity, we first consider the case of $\mathcal{K} := \{ u_n \}_{n=1}^{\infty} \cup u_0$, where $u_n \to u_0$ in $H^s_{qp}.$
Let $\epsilon > 0.$ 
Now, for any fixed $u_n$, there exists $M(n)$ such that for any $\delta > M(n), \| (u_n)_\delta - u_n\|_{H^s_{qp}} < \epsilon/3$.
Of all such $M(n),$ write $M_n$ for the minimal such value.

Furthermore, there exists $N$ such that for all $n > N$, $\| u- u_n \|_{H^s_{qp}} < \epsilon/3.$
Thus, for a fixed $n>N,$
\begin{align*}
	\| u_\delta - u \|_{H^s_{qp}}	
	&\leq \| u_\delta- (u_n)_\delta \|_{H^s_{qp}} + \| (u_n)_\delta - u_n\|_{H^s_{qp}} + \| u_n - u_0\|_{H^s_{qp}} \\
	&< \epsilon,
\end{align*}
so that $\| u_\delta - u \|_{H^s_{qp}} < \epsilon$. 
In particular, since $M_0$ is the minimal value such that $\delta > M_0$ implies $\| u_\delta - u \|_{H^s_{qp}} < \epsilon$, we must have that $M_n \geq M_0$.

Similarly, for the same $n > N$ that we fixed above we have
\begin{align*}
	\| (u_n)_\delta - u_n \|_{H^s_{qp}}	
	&\leq \| (u_n)_\delta  - (u_0)_\delta \|_{H^s_{qp}} + \| (u_0)_\delta - u_0 \|_{H^s_{qp}} + \| u_0 - u_n \|_{H^s_{qp}}\\
	&< \epsilon,
\end{align*}
so that $\| (u_n)_\delta - u_n \|_{H^s_{qp}} < \epsilon.$
Since $M_n$ is the minimal value for which $\delta > M_n$ implies $\| (u_n)_\delta - u_n \|_{H^s_{qp}} < \epsilon,$ we must have $M_0 \geq M_n$.
We thus obtain that $M_n = M_0$ for $n > N$, and therefore the value 
\[
	M = \sup_{i = 0,1,2,\ldots} M_i = \max_{i = 0,1, \ldots n} M_i
\]
must be finite.
Therefore, for any $\epsilon >0$, we can pick $M$ such that for any $u \in \mathcal{K}$ and $\delta > M,$ we have
\[ \| u_\delta - u\|_{H^s_{qp}} < \epsilon.
\] 

Now, we consider any compact set $\mathcal{K}.$ 
Let $\epsilon > 0,$ and for a given $u$ define $M_u$ to be the minimal value such that $\delta > M_u$ implies $\| u_\delta - u \|_{H^s_{qp}} < \epsilon$.
Define
\[ M = \sup_{u \in \mathcal{K}} M_u. \]
To show uniformity, we need that $M < \infty.$

Suppose not, i.e. $M = \infty$. 
Then, there exists a sequence $\{ u_n \} \subset \mathcal{K}$ such that $M_n := M_{u_n} > n$ for each $n$. 
By compactness of $\mathcal{K}$, there is a convergent subsequence $\{ u_{n_j} \} \subset \{u_{n}\}$. 
Furthermore, we know that there is a limit of this sequence $u_{n_0} \in \mathcal{K}.$ 
Let 
\[ 
D = \sup_{n_j: j = 0, 1, \ldots} M_{n_j}.
\] 
Since $M_{n_j} > n_j$, we must have that $D = \infty$.

However, the proof of the simple case indicates that $D < \infty$, hence a contradiction. 
Therefore, $M$ is bounded, and we obtain the desired uniformity on $\mathcal{K}.$
\end{proof}
\subsection{Existence of approximate solutions}
We mirror the approach of \cite[Section 8.5]{Ambrose_2016} and use the following Picard theorem to obtain solutions to the regularized equation \eqref{LWPeq1}.
\begin{thm}[Picard Theorem]\label{PicardTheorem}
Let $B$ be a Banach space and let $O\subseteq B$ be an open set. Let $F:O \to B$ be such that $F$ is locally Lipschitz: for any $X \in O$, there exists $\lambda > 0$ and an open set $U \subseteq O$ such that for all $Y, Z \in U,$
\[
	\| F(Y) - F(Z) \|_B \leq \lambda \| Y - Z \|_B. 
\]
Then, for all $X_0$, there exists $T>0$ and a unique $X \in C^1([-T, T]; O)$ such that $X$ solves the initial value problem 
\[
\frac{\D X}{\D t} = F(X), X(0) = X_0.	
\]
\end{thm}

\begin{thm}\label{LocalExistenceBad}
Let $s > N/2 + 1$. Let $u_0 \in H^{s}_{qp}.$
For any integer $n \geq 1$, there exists a $T_n>0$ and $u_n \in C([0,T_n]; H^{s}_{qp})$ such that $u_n$ satisfies the regularized Benjamin-Ono equation, as well as $u_n(0) = u_0$.
\end{thm}
\begin{proof}[Sketch]
Take $O = B= H^{s}_{qp}$ and define $F: H^{s}_{qp} \to H^{s}_{qp}$ by $F(u) = \chi_n[(\chi_n u) (\chi_n u_{x})] + \chi_n[H u_{xx}].$
The presence of mollifiers $\chi_n$ ensures that $F(u) \in H^s_{qp}$ for any $u \in H^s_{qp}$ and that $F$ is locally Lipschitz. 
By the Picard theorem, for each $n$, there is a $T_n$ such that there is a unique $u_n$ that solves the regularized equation and $u_n \in C^{1}([0,T_n]; H^{s}_{qp})$.
\end{proof}

\begin{rmk} In a typical periodic energy method proof, the approximate solutions may be finite dimensional, such as by having performed a Galerkin
projection.  In the present setting, our sum is still an infinite sum, as $\chi_{n}$ cuts off Fourier modes with $\alpha\cdot k$ large, rather than $k$ large.
It is well known that $\alpha\cdot k$ can accumulate near zero for large $k.$  However, even with this being an infinite sum, the point is that $\alpha\cdot k$
is now bounded, and thus norms of derivatives of a function $\chi_{n}f$ are bounded by norms of $f.$
\end{rmk}

\subsection{Uniform bound}
In this section, we seek to prove a bound on solutions $u_n$ derived in the previous section that is uniform in $n$. 
Having this bound would give us an existence of a common time $T$ on which the approximate solutions exist.
Crucially, $T$ will not depend on $n.$

Clearly, we know that 
\[ 
	\| u_n \|_{H^s_{qp}} \lesssim \| u_n \|_{L^2_{qp}} + \| D^s_x u_n \|_{L^2_{qp}},
\]
so it is enough to prove that $\| u_n \|_{L^2_{qp}}$ and $\| D_x^s u_n \|_{L^2_{qp}}$ are bounded uniformly.
Therefore, we separate our proof into two parts, to tackle each norm.

\begin{prop}[The $L^2_{qp}$ bound]\label{UB-prop-L2}
Let $u_n$ be the solution of the regularized Benjamin-Ono equation with initial data $u_{0}.$  Then we have the following conserved quantity:
\[
	\| u_n \|_{L^2_{qp}} = \| u_0 \|_{L^2_{qp}}.
\]
\end{prop}
\begin{proof}
For simplicity, write $u=u_n.$
The local existence result, Theorem \ref{LocalExistenceBad}, indicates that $u \in C^1$ in time.
Therefore, by Proposition \ref{ExchangeDerivandInt}, we can put the time derivative inside the norm:
\[ 
\frac{\D}{\D t} \| u \|_{L^2_{qp}}^2 = \lim_{M \to \infty} \frac{1}{2M} \int^M_{-M} 2 u_t u dx = 2 \sum_{k \in \ZZ^N} (\partial_t \hat{u})(k) \hat{u}(-k).
\]
Using the equation we obtain:
\begin{align*}
	\sum_{k \in \ZZ^N} (\partial_t \hat{u}(k)) \hat{u}(-k) = \sum_{k \in \ZZ^N} \FT{\chi_n((\chi_n u) (\chi_n u_x))}(k) \hat{u}(-k) + \sum_{k \in \ZZ^N} \FT{(\chi_n Hu_{xx})}(k) \hat{u}(-k)
\end{align*}
For the nonlinear term, note that 
\[
	\sum_{k \in \ZZ^N} \FT{\chi_n((\chi_n u) (\chi_n u_x))}(k) \hat{u}(-k) = \sum_{k \in \ZZ^N} \FT{\chi_n(u) \chi_n u_{x}}(k) \FT{\chi_n u}(-k),
\]
where we transfer the $\chi_n$ to $\hat{u}(-k)$. 
Applying Lemma \ref{Appendix2Id1} with $n=2$ we obtain that
\[ 
	\sum_{k \in \ZZ^N} \FT{\chi_n(u) \chi_n u_{x}}(k) \FT{\chi_n u}(-k) = \sum_{k \in \ZZ^N} \FT{\chi_n u_{x}} \FT{(\chi_n u)^2}(-k),
\]
vanishes. 

As for the linear term, since $\chi_n = (\chi_n)^2,$ we can write 
\[ 
	\sum_{k \in \ZZ^N} \FT{(\chi_n Hu_{xx})}(k) \hat{u}(-k) = \sum_{k \in \ZZ^N} \FT{(\chi_n Hu_{xx})}(k) \FT{\chi_n u}(-k)
\]
and apply Lemma \ref{lineartermvanishes} to conclude that the linear term vanishes.

Thus, we see that 
\[
	\frac{\D}{\D t} \| u \|_{L^2_{qp}}^2 = 0,
\]
which implies that $\| u \|_{L^2_{qp}}^2 = \text{constant} = \| u_0 \|_{L^2_{qp}}^2$.
\end{proof}

\begin{prop}\label{UB-prop}
For $s > N/2 + 1$, the solutions of each regularized problem satisfy the following bound:
\[
	\| D_x^s u_n \|_{L^2_{qp}} \leq ((\| D^{s}_x u_0\|_{L^2_{qp}} + \| u_0\|_{L^2_{qp}})^{-1} - Ct)^{-1},
\]
independently of $n$, for a fixed $C = C(s)$.
This bound holds if $C^{-1}(\| D^s_x u_0 \|_{L^2_{qp}} + \| u_0 \|_{L^2_{qp}})^{-1} > t$.
\end{prop}
\begin{proof}
As before, Proposition \ref{ExchangeDerivandInt} allows to estimate the time derivative of the Sobolev norm:
\begin{align*}
	\frac{\D \| D_x^s u \|_{L^2_{qp}}^2}{\D t} &= 2 \sum_{k \in \ZZ^N} \FT{D_x^s u}_t(k) \FT{D_x^s u}(-k) \\
	&= 2 \sum_{k \in \ZZ^N} \FT{(D^s_x [\chi_n[(\chi_n u) (\chi_n u_{x})] + \chi_n[H u_{xx}])}(k) \FT{D_x^s u}(-k) \\
	&= 2 \sum_{k \in \ZZ^N} \FT{D^s_x \chi_n[(\chi_n u) (\chi_n u_{x})]}(k) \FT{D_x^s u}(-k)  \\
	&\qquad \qquad + 2 \sum_{k \in \ZZ^N} \FT{ D^s_x \chi_n[H u_{xx}]}(k) \FT{D_x^s u}(-k),
\end{align*}
where we have used the regularized evolution equation, and we have split the sum into linear and nonlinear parts.
Since operators $D^s_x$ and $\chi_n H \partial_x^2$ commute, we can rewrite the linear term as
\begin{align*}
	\sum_{k \in \ZZ^N} \FT{ D^s_x \chi_n[H u_{xx}]}(k) \FT{D_x^s u}(-k) &=  \sum_{k \in \ZZ^N} \FT{\chi_n[H\partial_x^2  D^s_x u]}(k) \FT{D_x^s u}(-k) \\
	&= \sum_{k \in \ZZ^N} \FT{\chi_n[H\partial_x^2  D^s_x u]}(k) \FT{\chi_n(D_x^s u)}(-k), 
\end{align*}
so by Lemma \ref{lineartermvanishes}, this term is 0.

We now deal with the nonlinear term. To simplify the notation, let $v = \chi_n u$ and write
\begin{align*}
	\sum_{k \in \ZZ^N} \FT{D^s_x \chi_n[(\chi_n u) (\chi_n u_{x})]}(k) \FT{D_x^s u}(-k) &= 	\sum_{k \in \ZZ^N} \FT{D^s_x[(\chi_n u) (\chi_n u_{x})]}(k) \FT{D_x^s \chi_n u}(-k) \\
	&=\sum_{k \in \ZZ^N} \FT{D^s_x(vv_x)}(k) \FT{D_x^s v}(-k).
\end{align*}
Add and subtract $v D^s_x v$ to obtain:
\[ 
	\sum_{k \in \ZZ^N} \FT{D^s_x(vv_x) - v D^s_x v_x}(k) \FT{D_x^s v}(-k) + \sum_{k \in \ZZ^N} \FT{v D^s_x v_x}(k) \FT{D_x^s v}(-k) := I + II.
\]
The term $I$ is easily dealt with by applying the Cauchy-Schwarz inequality and the fractional Leibniz rule (Proposition \ref{FractionalLeibniz}):
\begin{align*}
	I &\leq \| D^s_x(vv_x) - v D^s_x v \|_{L^2_{qp}} \| D_x^s v \|_{L^2_{qp}} \\
	&\lesssim \big(\| v \|_{H^{s}_{qp}} \| v_x\|_{H^{s_0}_{qp}} + \| v \|_{H^{s_0+1}_{qp}} \| v_x\|_{H^{s-1}_{qp}}\big) \| D_x^s v \|_{L^2_{qp}},
\end{align*}
for some $s_0 > N/2.$
In particular, letting $s_0 = s-1 >N/2$ yields  
\[
	I \lesssim \big( \| v \|_{H^{s}_{qp}} \| v_x\|_{H^{s-1}_{qp}} + \| v \|_{H^{s-1+1}_{qp}} \| v_x\|_{H^{s-1}_{qp}} \big) \| D_x^s v \|_{L^2_{qp}}\leq \| v \|_{H^{s}_{qp}}^2 \| D_x^s v \|_{L^2_{qp}}.
\] 
As for $II$, we first apply integration by parts from Lemma \ref{IntegrationByParts}, yielding
\[
	II  = \sum_{k \in \ZZ^N} \FT{v D^s_x v_x}(k) \FT{D_x^s v}(-k) = - \frac{1}{2}\sum_{k \in \ZZ^N} \FT{v_x D_x^s v}(k) \FT{D_x^s v}(-k).
\]
We can then apply the Cauchy-Schwarz, Young's, and Sobolev inequalities to obtain
\begin{align*}
	II \lesssim \sum_{k \in \ZZ^N} |\FT{v_x D_x^s v}(k) \FT{D_x^s v}(-k)|
	&\lesssim \|\FT{v_x D_x^s v} \|_{\ell^2_k } \| \FT{D_x^s v} \|_{\ell^2_k}\\
	&\lesssim \|\FT{v_x}\|_{\ell^1_k}\| \FT{D_x^s v} \|_{\ell^2_k}^2 \\
	&\lesssim \| v \|_{H^{s}_{qp}}\| D^s_x v\|_{L^2_{qp}}^2.
\end{align*} 
All in all, we see that the nonlinear term is bounded by 
\[
I + II \lesssim \| v \|_{H^{s}_{qp}}\| D^s_x v\|_{L^2_{qp}}^2 + \| v \|_{H^{s}_{qp}}^2 \| D^s_x v\|_{L^2_{qp}}.
\]
Finally, note that by Proposition \ref{UB-prop-L2}, $\| v\|_{L^2_{qp}}$ is bounded; thus, we have
\[ 
\| v \|_{H^s_{qp}} \lesssim  \| v\|_{L^2_{qp}} + \| D^s_x v \|_{L^2_{qp}} \lesssim \| v_0 \|_{L^2_{qp}} + \| D^s_x v \|_{L^2_{qp}}.
\]
As a result,
\begin{align*}
	I + II &\lesssim \| v \|_{H^{s}_{qp}}\| D^s_x v\|_{L^2_{qp}}^2 + \| v \|_{H^{s}_{qp}}^2 \| D^s_x v\|_{L^2_{qp}} \\
	&\leq
	(\| u_0 \|_{L^2_{qp}} + \| D^s_x v \|_{L^2_{qp}})\| D^s_x v\|_{L^2_{qp}}^2 + (\| u_0 \|_{L^2_{qp}}^2 + \| D^s_x v \|_{L^2_{qp}}^2)\| D^s_x v\|_{L^2_{qp}} \\
&\lesssim \| D^s_x v\|_{L^2_{qp}}^3 + \| u_0 \|_{L^2_{qp}} \| D^s_x v\|_{L^2_{qp}}^2 + \| u_0 \|_{L^2_{qp}}^2 \| D^s_x v\|_{L^2_{qp}} \\
&\lesssim \| D^s_x v\|_{L^2_{qp}} (\| D^s_x v\|_{L^2_{qp}} + \| v_0\|_{L^2_{qp}})^2,
\end{align*}
so that
\[
\frac{\D \| D^s_x u_n \|_{L^2_{qp}}^2}{\D t} = I + II \lesssim \| D^s_x v\|_{L^2_{qp}} (\| D^s_x v\|_{L^2_{qp}} + \| v_0\|_{L^2_{qp}})^2.
\]
Thus, 
\[
\frac{\D \| D^s_x u_n \|_{L^2_{qp}}}{\D t} \leq C (\| D^s_x v\|_{L^2_{qp}} + \| v_0\|_{L^2_{qp}})^2,
\]
where $C$ depends on $s.$ 
A Gronwall-type argument then reveals that $\| D^s_x u_n \|_{L^2_{qp}}$ is bounded, independently of $n$:
\[ 
	\| D^s_x u_n \|_{L^2_{qp}} \leq ((\| D^{s}_x u_0\|_{L^2_{qp}} + \| v_0\|_{L^2_{qp}})^{-1} - Ct)^{-1}.	
\]
This bound holds whenever $C^{-1}(\| D^s_x u_0 \|_{L^2_{qp}} + \| v_0\|_{L^2_{qp}})^{-1} > t$.
\end{proof}

We have now established the uniform bound and the existence of a common time interval, independent of $n.$ 
\begin{cor}[Uniform bound for the regularized problem]\label{UniformBound}
Let $s > N/2 +1.$ 
Let $u_0 \in H^s_{qp}.$ 
Then, solutions $u_{n}$ to the regularized Benjamin-Ono equation \eqref{LWPeq1} with initial data $u_0$ satisfy the following bound independent of parameters $n:$
\begin{equation}\label{UB-prop-est2}
	\| u_{n} \|_{H^s_{qp}} \leq \| u_{n} \|_{L^2_{qp}} + \| D_x^s u_{n} \|_{L^2_{qp}} \lesssim \| u_0 \|_{H^s_{qp}} + ((\| u_0 \|_{H^s_{qp}})^{-1} - Ct)^{-1},
\end{equation}
Furthermore, these solutions exist on a common time interval $T$ that only depends on $s$ and the norm $\| u_0 \|_{H^s_{qp}}.$
In particular, if the initial data is chosen from a compact set $\mathcal{K} \subseteq H^s_{qp}$, then $T$ only depends on $\mathcal{K}$ and not a particular choice of the initial data.
\end{cor}
%
%

At last, the result below allows us to estimate solutions in higher order Sobolev spaces.
\begin{prop}[Refined boundedness]\label{refinedbound}
Let $s > N/2 + 1$. Let $\delta > 1$.
Suppose $u_{n,\delta}$ solves the regularized Benjamin-Ono equation with parameter $n$ and regularized initial data $(u_0)_\delta,$ as defined in \eqref{regData}.
Further, let $T>0$ be such that these solutions exist on the same interval $[0,T].$
Then, $u_{n,\delta}$ satisfy the following estimate:
\[ 
\| u_{n,\delta} \|_{H^{s+l}_{qp}} \lesssim \delta^l \| u_0\|_{H^s_{qp}},
\]
for any $l >0$, for any $t \in [0,T].$
\end{prop}
\begin{proof}
For simplicity, write $u = \chi_n u_{n,\delta}$ and take the time derivative of $\| D^{s + l}_x u_{n,\delta} \|_{L^2_{qp}}^2$.
Following the proof of the uniform bound yields
\begin{align*}
	\frac{\D \| D_x^{s+l} u_{n,\delta} \|_{L^2_{qp}}^2}{\D t}
	&= 2 \sum_{k \in \ZZ^N} \FT{D^{s+l}_x \chi_n[(\chi_n u_{n,\delta}) (\chi_n (u_{n,\delta})_{x})]}(k) \FT{D_x^{s+l} u_{n,\delta}}(-k)  \\
	&\qquad \qquad + 2 \sum_{k \in \ZZ^N} \FT{ D^{s+l}_x \chi_n[H (u_{n,\delta})_{xx}]}(k) \FT{D_x^{s+l} u_{n,\delta}}(-k) \\
	&= 2 \sum_{k \in \ZZ^N} \FT{D^{s+l}_x [u  u_x]}(k) \FT{D_x^{s+l} u}(-k).
\end{align*}
The linear term vanishes as before, while for the nonlinear term we add and subtract $u D^{s+l}_x u_x$ to get :
\begin{align*}
	\sum_{k \in \ZZ^N} \FT{D^{s+l}_x [u  u_x] - u D^{s+l}_x u_x}(k) \FT{D_x^{s+l} u}(-k) + \sum_{k \in \ZZ^N} \FT{u D^{s+l}_x u_x}(k) \FT{D_x^{s+l} u}(-k) =: I +II.
\end{align*}
For the sum $I$, apply the Cauchy-Schwarz inequality and use Proposition \ref{FractionalLeibniz} with $s_0 = s- 1$ to get
\begin{align*} 
	&\left| \sum_{k \in \ZZ^N} \FT{D^{s+l}_x [u  u_x] - u D^{s+l}_x u_x}(k) \FT{D_x^{s+l} u}(-k) \right|\\
	&\qquad \qquad \leq \| D^{s+l}_x [u  u_x] - u D^{s+l}_x u_x \|_{L^2_{qp}} \| D_x^{s+l} u\|_{L^2_{qp}} \\
	&\qquad \qquad \leq \left( \| u \|_{H^{s+l}_{qp}} \| u_x \|_{H^{s-1}_{qp}} + \| u \|_{H^{s}_{qp}}  \| u_x \|_{H^{s+l-1}_{qp}} \right)\| D_x^{s+l} u\|_{L^2_{qp}}.
\end{align*}
Note that $\| u_x \|_{H^{s-1}_{qp}}$ and $\| u \|_{H^{s}_{qp}}$ are bounded on $[0,T],$ so that 
\[
| I | \lesssim_T  \| u \|_{H^{s+l}_{qp}} \| D_x^{s+l} u\|_{L^2_{qp}} \lesssim \| (u_0)_\delta \|_{L^2_{qp}} \| D_x^{s+l} u\|_{L^2_{qp}} +  \| D_x^{s+l} u\|^2_{L^2_{qp}}.
\]
As for $II$, we again apply integration by parts, finding
\begin{align*}
	II  = \sum_{k \in \ZZ^N} \FT{u D^{s+l}_x u_x}(k) \FT{D_x^{s+l} u}(-k) = - \frac{1}{2}\sum_{k \in \ZZ^N} \FT{u_x D_x^{s+l} u}(k) \FT{D_x^{s+l} u}(-k),
\end{align*}
so that
\[ II \lesssim \| u_x D_x^{s+l} u \|_{L^2_{qp}} \| D_x^{s+l} u \|_{L^2_{qp}} \lesssim \| \FT u_x \|_{\ell^1_k} \| D_x^{s+l} u \|_{L^2_{qp}}^2.
\]
Applying the Sobolev inequality (Proposition \ref{SobolevInequality}) and the uniform bound on $\| u \|_{H^s_{qp}}$ yields 
\[
	\| \FT u_x \|_{\ell^1_k} \lesssim \|u_x \|_{H^{s-1}_{qp}} \lesssim \|u \|_{H^{s}_{qp}} \lesssim_T 1.
\]
Combining the bounds on $I$ and $II$, we obtain 
\[ 
	\frac{\D \| D_x^{s+l} u_{n,\delta} \|_{L^2_{qp}}^2}{\D t} \lesssim \| D_x^{s+l} u_{n,\delta} \|_{L^2_{qp}}^2.
\] 
Finally, applying Gronwall's inequality and properties of data regularization \eqref{RD1} yields
\[
	\| D_x^{s+l} u_{n,\delta} \|_{L^2_{qp}}^2 \lesssim \| D_x^{s+l} (u_0)_\delta \|_{L^2_{qp}}^2 \lesssim \delta^{2l} \| D_x^{s} (u_0)_\delta \|_{L^2_{qp}}^2.
\]
We clearly have 
\[
	\| u_{n,\delta} \|_{L^2_{qp}} = \| (u_0)_\delta \|_{L^2_{qp}} \leq \delta^{l} \| (u_0)_\delta \|_{L^2_{qp}},
\]
so that we obtain the desired inequality.
\end{proof}
\subsection{Cauchy estimate}
In this section, we provide a Cauchy estimate for solutions of the regularized problem.
We first outline our approach and provide a few results to be used later in the proof.
Let $u_0 \in H^s_{qp}$ and let $(u_0)_\delta$ be a regularized version of $u_0$ as defined in \eqref{regData}, with parameter $\delta>0$.
Let $n,m \in \NN$. Let $u_{n, \delta_1}$ solve
\begin{align*}
	u_t &= \chi_n[(\chi_n u_n) (\chi_n u_{nx})] + \chi_n[H u_{nxx}],  \\
	u(0,\cdot) &= (u_0)_{\delta_1},
\end{align*}
whereas $u_{m, \delta_2}$ denotes the solution of 
\begin{align*}
	u_t &= \chi_m[(\chi_m u_m) (\chi_m u_{mx})] + \chi_m[H u_{mxx}], \\
	u(0,\cdot) &= (u_0)_{\delta_2}.
\end{align*}
For ease of notation, write $u = u_{n, \delta_1}, v = u_{m,\delta_2}, w = u-v.$ 
Note that the set $\mathcal{U} := \{ (u_0)_{\delta}: \delta \geq 0 \} \cup u_0$ is both a complete subset of $H^s_{qp}$ and totally bounded $H^s_{qp}$ by the uniform bound.
Hence $\mathcal{U}$ is compact, and so by Corollary \ref{UniformBound}, $u$ and $v$ exist on the same time interval.

Write the difference of equations:
\begin{align*}
	u_t - v_t &= \chi_n[(\chi_n u_n) (\chi_n u_{nx})] + \chi_n[H u_{nxx}] - (\chi_m[(\chi_m u_m) (\chi_m u_{mx})] + \chi_m[H u_{mxx}]) \\
	&= \chi_n[(\chi_n u_n) (\chi_n u_{nx})] - \chi_m[(\chi_m u_m) (\chi_m u_{mx})] + \chi_n[H u_{nxx}] - \chi_m[H u_{mxx}].
\end{align*}
The linear term can be written as follows:
\begin{equation}\label{linearterm}
\begin{aligned}
	\chi_n[ H u_{xx}] - \chi_m[Hv_{xx}] &= \chi_n[ H u_{xx}] - \chi_m[ H u_{xx}] +\chi_m[ H u_{xx}] - \chi_m[Hv_{xx}] \\
	&= \chi_m[ H w_{xx}] +[\chi_n - \chi_m]Hu_{xx}.
\end{aligned}
\end{equation}
For the nonlinear term, observe the following equality:
\begin{align*}
	\chi_m[ (\chi_m w) (\chi_m w_x)] = \chi_m[ (\chi_m u) (\chi_m u_x)] &- \chi_m[(\chi_m u) (\chi_m v_x)] \\
	&- \chi_m[(\chi_m v) (\chi_m u_x)] + \chi_m[(\chi_m v) (\chi_m v_x)],
\end{align*}
so that 
\begin{align*}
	\chi_m[(\chi_m v) (\chi_m v_x)] =  \chi_m[ (\chi_m w) (\chi_m w_x)] &+ \chi_m[(\chi_m u) (\chi_m v_x)] \\
	&+ \chi_m[(\chi_m v) (\chi_m u_x)] - \chi_m[ (\chi_m u) (\chi_m u_x)] \\
	= \chi_m[ (\chi_m w) (\chi_m w_x)] &+ \chi_m[(\chi_m u) (\chi_m v_x)] - \chi_m[(\chi_m w) (\chi_m u_x)].
\end{align*}
As such it is straightforward to verify the following identity:
\begin{equation}\label{nonlinearterm}
\begin{aligned}
	\chi_n[ (\chi_n u) (\chi_n u_x)] &- \chi_m[ (\chi_m v) (\chi_m v_x)] \\
	&= \chi_m[ (\chi_m v) (\chi_m w_x)] + \chi_m[(\chi_m w)(\chi_m u)_x]+ \chi_m[(\chi_m u) ((\chi_n - \chi_m) u)_x] \\
	&+ \chi_m[ ((\chi_n - \chi_m)u) (\chi_n u)_x] + [\chi_n - \chi_m][\chi_n u \chi_n u_x].
\end{aligned}
\end{equation}
As is seen in both \eqref{linearterm} and \eqref{nonlinearterm}, there are terms with factors $\chi_n - \chi_m$.
The lemma below explains how to handle these terms.

\begin{lem}\label{DifferenceEst}
Let $n, m\in \NN$ and $l \geq 0.$ For any $v \in H^l_{qp},$ we have 
\[
\| (\chi_n - \chi_m) v \|_{L^2_{qp}} \leq \max\{1/n, 1/m\}^{l} \| D^{l}_x v \|_{L^2_{qp}}.
\]
\end{lem}
\begin{proof}
Note that
\[ 
	\| (\chi_n - \chi_m) v \|_{L^2_{qp}}^2 = \sum_{k} |\hat{v}|^2(k) \mathbb{I}_{\min\{n,m\} \leq |\alpha \cdot k| \leq \max\{n,m\}}.
\]
Since $\min\{n,m\} \leq |\alpha \cdot k| \leq |k|$, we know that $|k|^{-1} \leq \max\{1/n, 1/m\}$.
Thus, 
\begin{align*}
	\| (\chi_n - \chi_m) v \|_{L^2_{qp}}^2 &= \sum_{k} |\hat{v}|^2(k) \mathbb{I}_{\min\{n,m\}(k) \leq |\alpha \cdot k| \leq \max\{n,m\}} \\
	&= \sum_{k} \mathbb{I}_{\min\{n,m\} \leq |\alpha \cdot k| \leq \max\{n,m\}}(k) |\hat{v}|^2(k) |k|^{-2l} |k|^{2l} \\
	&\leq \big( \max\{1/n, 1/m\}\big)^{2l} \sum_{k} \mathbb{I}_{\min\{n,m\} \leq |\alpha \cdot k| \leq \max\{n,m\}}(k) |\FT{\partial_x^l v}|^2(k) \\
	&\leq \big( \max\{1/n, 1/m\}\big)^{2l}  \| D_x^l v \|_{L^2_{qp}}^2,
\end{align*}
from which the lemma follows.
\end{proof}
We are now ready to prove Cauchy estimates. 
We begin with the $L^2_{qp}$ case.
\begin{prop}
Let $m, n \in \NN$. Then, solutions $u = u_{n, \delta_1}, v = u_{m, \delta_2}$ be defined as above. Then, 
\[
	\frac{\D}{\D t} \| u - v\|_{L^2_{qp}}^2 \lesssim \| u - v\|_{L^2}^2 + \max\{1/n, 1/m\}^{s-2} \| u - v\|_{L^2},
\]
where implicit constants do not depend on $n,m$, but may depend on the uniform bound derived before. 
In particular, we have
\begin{equation}\label{L2CauchyEstimate}
	\| u-v \|_{L^2_{qp}} \lesssim \max\{1/n, 1/m\}^{s-2} + \|(u_0)_{\delta_1} - (u_0)_{\delta_2} \|_{L^2_{qp}}.
\end{equation}
\end{prop}
\begin{proof}
As before, we take the time derivative of the $L^2_{qp}$ norm:
\begin{align*}
	\frac{\D \| w \|_{L^2_{qp}}^2}{\D t} &= 2 \sum_{k} \FT{(u-v)_t}(k) \hat{w}(-k) \\
	&= 2 \sum_{k} \FT{\chi_n[(\chi_n u_n) (\chi_n u_{nx})] - \chi_m[(\chi_m u_m) (\chi_m u_{mx})]}(k) \hat{w}(-k) \\
	&\qquad + 2 \sum_{k} \FT{\chi_n[H u_{nxx}] - \chi_m[H u_{mxx}]}(k) \hat{w}(-k).
\end{align*}
Following the analysis above, the linear part can be estimated following \eqref{linearterm}.
Thus, 
\begin{align*}
	\sum_{k \in \ZZ^N} &\FT{(\chi_n[ H u_{xx}] - \chi_m[Hv_{xx}])}(k) \FT w(-k) \\
	&= \underbrace{\sum_{k \in \ZZ^N}\FT{(\chi_m[ H w_{xx}])}(k) \FT w(-k)}_\text{$=0$ by Lemma \ref{lineartermvanishes}} 
	+ \sum_{k \in \ZZ^N}\FT{([\chi_n - \chi_m]Hu_{xx})}(k) \FT w(-k).
\end{align*}
By applying Cauchy-Schwarz inequality, it is enough to estimate 
\[ 
\| \FT{D^{2}_x ([\chi_n - \chi_m]u)}\|_{\ell^2_k}.
\]
By Lemma \ref{DifferenceEst}, we can bound 
\begin{align*}
	\| \FT{D^{2}_x ([\chi_n - \chi_m]v)}\|_{\ell^2_k} &\lesssim \max\{1/n, 1/m\}^{l_0}\| u\|_H^{2+l_0}.
\end{align*}
Since $s > N/2 +1$ and $N>1$, we can write $s = 2 + \epsilon$ and let $l_0= \epsilon.$ 
With this choice, the linear term is bounded by 
\[ 
	\max\{1/n, 1/m\}^{\epsilon}\| u\|_H^{s} \|w \|_{L^2_{qp}} \lesssim \max\{1/n, 1/m\}^{\epsilon} \|w \|_{L^2_{qp}},
\]
up to constants that depend on $s$ and the uniform bound derived earlier.

We now deal with the nonlinear term.
By decomposition \eqref{nonlinearterm}, we need to estimate
\begin{align*}
	&\sum_k \FT{\chi_m[ (\chi_m v) (\chi_m w_x)]}(k) \FT{w}(-k) + \sum_k \FT{\chi_m[(\chi_m w)(\chi_m u)_x]}(k) \FT{w}(-k) \\
	&+ \sum_k \FT{\chi_m[(\chi_m u) ((\chi_n - \chi_m) u)_x]}(k) \FT{w}(-k) + \sum_k \FT{\chi_m[ ((\chi_n - \chi_m)u) (\chi_n u)_x]}(k) \FT{w}(-k) \\
	&+ \sum_k \FT{[\chi_n - \chi_m][\chi_n u \chi_n u_x]}(k) \FT{w}(-k) \\
	&= I + II + III + IV + V.
\end{align*}
We seek to bound each term:
\begin{enumerate}
	\item the term $I$ is handled by energy cancellation:
\begin{align*}
	\sum_k \FT{\chi_m[ (\chi_m v) (\chi_m w_x)]}(k) \FT{w}(-k) &= \sum_k \FT{(\chi_m v) (\chi_m w_x)}(k) \FT{\chi_m w}(-k) \\
	&= -\frac{1}{2} \sum_k \FT{(\chi_m v)_x \chi_m w}(k) \FT{\chi_m w}(-k),
\end{align*}
Then, apply Cauchy-Schwarz inequality, Young's convolution inequality, and Sobolev embedding as follows:
\begin{align*}
	I &\lesssim \| \FT{(\chi_m w)(\chi_m v)_x}\|_{\ell^2_k} \| \chi_m w \|_{L^2_{qp}} \\
	&\leq \|\FT{(\chi_m v)_x}\|_{\ell^1_k} \| w \|_{L^2_{qp}}^2 \\
	&\lesssim \| v \|_{H^s_{qp}} \| w \|_{L^2_{qp}}^2,
\end{align*}
so that $I$ is bounded by $\| w \|_{L^2_{qp}}^2$.
	\item the term $II$ is handled similarly to $I$: 
\begin{align*}
	\sum_k \FT{\chi_m[(\chi_m w)(\chi_m u)_x]}(k) \FT{w}(-k) &= \sum_k \FT{[(\chi_m w)(\chi_m u)_x]}(k) \FT{\chi_m w}(-k)\\
	&\leq \| \FT{(\chi_m w)(\chi_m u)_x}\|_{\ell^2_k} \| \chi_m w \|_{L^2_{qp}} \\
	&\leq \|\FT{(\chi_m u)_x}\|_{\ell^1_k } \| w \|_{L^2_{qp}}^2 \\
	&\lesssim \| u \|_{H^s_{qp}} \| w \|_{L^2_{qp}}^2,
\end{align*}	
so that $II$ is bounded by $\| w \|_{L^2_{qp}}^2$.
	\item the terms $III$ and $IV$ are handled similarly to $I$ and $II$ except that Lemma \ref{DifferenceEst} is invoked to gain decay. 
For example, let us consider $III:$
\begin{align*}
	\sum_k \FT{\chi_m[(\chi_m u) ((\chi_n - \chi_m) u)_x]}(k) \FT{w}(-k) &\leq \| \FT{(\chi_m u) ((\chi_n - \chi_m) u)_x}\|_{\ell^2_k} \| w \|_{L^2_{qp}} \\
	&\leq \| \FT{(\chi_n - \chi_m) u_x} \|_{\ell^2_k} \| \FT{\chi_m u} \|_{\ell^1_k} \| w \|_{L^2_{qp}} \\
	&\lesssim \max\{ 1/n, 1/m\}^{s-1}\| u \|_{H^s_{qp}} \| u \|_{H^{s-1}_{qp}}\| w \|_{L^2_{qp}} \\
	&\lesssim \max\{ 1/n, 1/m\}^{s-1} \| w \|_{L^2_{qp}}.
\end{align*}
In the same way, we obtain
\[ 
IV \lesssim \max\{ 1/n, 1/m\}^{s} \| w \|_{L^2_{qp}}.
\]
\item for the term $V$, apply Cauchy-Schwarz inequality and Lemma \ref{DifferenceEst}:
\begin{align*}
	\sum_k \FT{[\chi_n - \chi_m][\chi_n u \chi_n u_x]}(k) \FT{w}(-k) &\leq \| [\chi_n - \chi_m][\chi_n u \chi_n u_x] \|_{L^2_{qp}} \| w \|_{L^2_{qp}} \\
	&\lesssim \max\{ 1/n, 1/m\}^{s-1} \| \chi_n u \chi_n u_x \|_{H^{s-1}_{qp}} \| w \|_{L^2_{qp}} \\
	&\lesssim \max\{ 1/n, 1/m\}^{s-1} \| \chi_n u \|_{H^{s-1}_{qp}} \| \chi_n u_x \|_{H^{s-1}_{qp}} \| w \|_{L^2_{qp}},
\end{align*}
where we apply the algebra property in the last line.
Thus, $V \leq \max\{ 1/n, 1/m\}^{s-1} \| w \|_{L^2_{qp}}.$
\end{enumerate}
Combining these estimates with the linear part yields the following inequality:
\begin{align*}
	\frac{\D \| w \|_{L^2_{qp}}^2}{\D t} &\leq \max\{1/n, 1/m\}^{\epsilon} \|w \|_{L^2_{qp}} + \|w \|_{L^2_{qp}}^2 \\
	&+ \max\{ 1/n, 1/m\}^{s} \| w \|_{L^2_{qp}} + \max\{ 1/n, 1/m\}^{s-1} \| w \|_{L^2_{qp}}.
\end{align*}
Note that since $s = 2 + \epsilon$, we have $\epsilon = s-2$. In particular, this will be the slowest decaying factor, so that 
\[ 
	\frac{\D \| w \|_{L^2_{qp}}^2}{\D t} \lesssim \max\{1/n, 1/m\}^{s-2} \|w \|_{L^2_{qp}} + \|w \|_{L^2_{qp}}^2,
\]
which becomes 
\[ 
	\frac{\D \| w \|_{L^2_{qp}}}{\D t} \lesssim \max\{1/n, 1/m\}^{s-2} + \|w \|_{L^2_{qp}}.
\]
A standard application of the Gronwall's inequality now yields
\[ 
	\| w \|_{L^2_{qp}} \lesssim e^t ( \max\{1/n, 1/m\}^{s-2} + \|w_0 \|_{L^2_{qp}}),
\]
which completes the baseline estimate.
\end{proof}

An easy consequence of the $L^2_{qp}$ Cauchy estimate and the interpolation inequality (Lemma \ref{InterpolationInequality}) 
is that we obtain Cauchy property of solutions for $H^l_{qp}$ where $0 \leq l <s.$
In the particular case of $l = s, p = s-1$, we have
\[ 
\| w\|_{H^{s-1}_{qp}} \leq \|w\|_{H^{l}_{qp}}^{(s-1)/s} \|w\|_{L^2_{qp}}^{1-(s-1)/s} = \|w\|_{H^{l}_{qp}}^{(s-1)/s} \|w\|_{L^2_{qp}}^{1/s}
\]
Due to triangle inequality and the uniform bound on $u$ and $v$, we can ignore $\|w\|_{H^{l}_{qp}}^{(s-1)/s}.$
Now, invoking \eqref{L2CauchyEstimate} we have
\begin{align*}
	\| w\|_{H^{s-1}_{qp}} &\leq \Big( \max\{1/n, 1/m\}^{s-2} + \|(u_0)_{\delta_1} - (u_0)_{\delta_2} \|_{L^2_{qp}} \Big)^{1/s} \\
	&\leq \max\{1/n, 1/m\}^{(s-2)/s} + \|(u_0)_{\delta_1} - (u_0)_{\delta_2} \|_{L^2_{qp}}^{1/s}.
\end{align*}
By \eqref{RD2}, we have
\[
	\|(u_0)_{\delta_1} - (u_0)_{\delta_2} \|_{L^2_{qp}} \leq \|(u_0)_{\delta_1} - u_0 \|_{L^2_{qp}} + \|(u_0)_{\delta_2} - u_0 \|_{L^2_{qp}} \lesssim o((\delta_1)^{-s}) + o((\delta_2)^{-s}),
\]
so that 
\begin{align*}
	\| w\|_{H^{s-1}_{qp}} \lesssim \max\{1/n, 1/m\}^{(s-2)/s} + (o((\delta_1)^{-s}) + o((\delta_2)^{-s}))^{1/s}.
\end{align*}
Thus, we obtain
\begin{equation}\label{lowerInterp}
	\| w\|_{H^{s-1}_{qp}} \lesssim \max\{1/n, 1/m\}^{(s-2)/s} + o((\delta_1)^{-1}) + o((\delta_2)^{-1}).
\end{equation}
We now would like to prove the highest order estimate. 
It is here that we will use good properties of regularized data \eqref{regData}.

\begin{prop}\label{HsCauchyEstimate}
Let $m, n \in \RR_{+}$ and $\delta_1, \delta_2 \in \RR_{+}$. 
Assume $\delta_1 < \delta_2$ and $m >n$. 
Furthermore, assume $\delta_1 \leq n^{(s-2)/s - \epsilon},$ for some $\epsilon>0.$ 
Let $u_{n, \delta_1}, u_{m, \delta_2}$ be solutions for the respective regularized Benjamin-Ono problem which exist on the same time interval $[0,T]$.
Then, we have the following estimate
\[
\| u_{n, \delta_1} - u_{m, \delta_2} \|_{L^\infty_T H^s_{qp}} = o(1)
\]
for large enough $\delta_1$. 
In addition, the implicit constants depend only on the norm of the initial data and as such can be taken uniform as long as the initial data comes from a compact set.
\end{prop}
\begin{proof}[Proof of Proposition \ref{HsCauchyEstimate}]
As before, we consider the time derivative of $\| D^s_x (u_{n,\delta_1} -u_{m,\delta_2})\|_{L^2_{qp}}^2.$
For ease of notation, write $u = u_{n, \delta_1}, v = u_{m,\delta_2}, w = u-v,$ so that we wish to estimate 
\begin{align*}
	\frac{\D}{\D t} \| D^s_x w \|_{L^2_{qp}}^2 &= 2 \sum_{k \in \ZZ^N} \FT{D^s_x w}(k) \FT D^s_x w(-k) \\
	&= 2 \sum_{k \in \ZZ^N}\FT{D^s_x (\chi_n[ (\chi_n u) (\chi_n u_x)] - \chi_m[ (\chi_m v) (\chi_m v_x)])} \FT D^s_x w(-k) \\
	&+\quad 2 \sum_{k \in \ZZ^N}\FT{D^s_x (\chi_n[ H u_{xx}] - \chi_m[Hv_{xx}])} \FT D^s_x w(-k).
\end{align*}
First, we pick the $\epsilon>0$. Note that since 
\[ 
	s > N/2 + 1 \geq 2/2 + 1 = 2,
\]
we know that there is some $\epsilon > 0$ such that
\[ 
	\frac{s-2}{s} > \epsilon > \frac{2}{3}\frac{s-2}{s}.
\]
Note that with this choice $\epsilon$, we necessarily have 
\[
	(s-2)/s - \epsilon >0 \qquad \text{and} \qquad 2(s-2)/s - 3\epsilon < 0.
\]
For convenience, we write
\begin{equation}\label{cond1}
	\delta_1 \leq n^{(s-2)/s - \epsilon}, \quad (s-2)/s - \epsilon >0, \quad 2(s-2)/s - 3\epsilon < 0,
\end{equation}
so that we can refer to these conditions easily.
Furthermore, note that 
\begin{equation}\label{cond2}
	\delta_1 n^{-1} \leq n^{(s-2)/s- \epsilon} n^{-1} = n^{(-s+s -2)/s - \epsilon} =n^{-2/s - \epsilon}.
\end{equation}
Now, we consider the linear part.
Due to \eqref{linearterm}, we have 
\begin{align*}
	\sum_{k \in \ZZ^N} &\FT{D^s_x (\chi_n[ H u_{xx}] - \chi_m[Hv_{xx}])} \FT D^s_x w(-k) \\
	&= \underbrace{\sum_{k \in \ZZ^N}\FT{D^s_x (\chi_m[ H w_{xx}])} \FT D^s_x w(-k)}_\text{$=0$ by Lemma \eqref{lineartermvanishes}} 
	+ \sum_{k \in \ZZ^N}\FT{D^s_x ([\chi_n - \chi_m]Hu_{xx})} \FT D^s_x w(-k).
\end{align*}
By applying Cauchy-Schwarz inequality, it is enough to estimate 
\[ 
\| \FT{D^{s+2}_x ([\chi_n - \chi_m]u)}\|_{\ell^2_k}.
\]
Due to the presence of $\chi_n - \chi_m$, we first apply Lemma \eqref{DifferenceEst} and then use Lemma \eqref{refinedbound}:
\begin{align*}
	\| \FT{D^{s+2}_x ([\chi_n - \chi_m]u)}\|_{\ell^2_k}
		&\lesssim \max\{1/n, 1/m\}^{l_1} \| u \|_{H^{s+{l_1}+2}_{qp}}\\
		&\lesssim \max\{1/n, 1/m\}^{l_1} (\delta_1)^{{l_1} + 2}\| u\|_{H^{s}_{qp}} \\
		&= n^{-l_1}(\delta_1)^{{l_1} + 2}\| u\|_{H^{s}_{qp}}.
\end{align*}
Thus, the linear term is bounded by 
\[
	n^{-l_1}(\delta_1)^{{l_1} + 2} \| D^s_x w \|_{L^2_{qp}},
\]
where the constants depend on $s$ and the uniform bound on $v$. Letting $l_1 = 1$ lets us bound the linear term by
\[ 
n^{-1} \delta_1^{3}\| D^s_x w\|_{L^2_{qp}}.
\]

We proceed with the nonlinear term. 
Using the decomposition in \eqref{nonlinearterm}, we have
\begin{align*}
	\sum_{k \in \ZZ^N} &\FT{D^s_x (\chi_n[ (\chi_n u) (\chi_n u_x)] - \chi_m[ (\chi_m v) (\chi_m v_x)])}(k) \FT D^s_x w(-k) \\
	&= \sum_{k \in \ZZ^N}\FT{D^s_x \chi_m[ (\chi_m v) (\chi_m w_x)]}(k) \FT D^s_x w(-k)
	+ \sum_{k \in \ZZ^N}\FT{D^s_x \chi_m[(\chi_m w)(\chi_m u)_x]}(k) \FT D^s_x w(-k) \\
	&+ \sum_{k \in \ZZ^N}\FT{D^s_x \chi_m[(\chi_m u) ((\chi_n - \chi_m) u)_x]}(k) \FT D^s_x w(-k) \\
	&+ \sum_{k \in \ZZ^N}\FT{D^s_x \chi_m[ ((\chi_n - \chi_m)u) (\chi_n u)_x]}(k) \FT D^s_x w(-k)
	+ \sum_{k \in \ZZ^N}\FT{D^s_x [\chi_n - \chi_m][\chi_n u \chi_n u_x]}(k) \FT D^s_x w(-k)\\
	&= I + II + III + IV + V.
\end{align*}
We seek to bound each term. 
The common strategy in all calculations is to add and subtract certain terms to use the fractional Leibniz rule (Proposition \ref{FractionalLeibniz}).
Note:
\begin{enumerate}
	\item the term $I$ has the same structure as the nonlinearity in the regularized Benjamin-Ono equation, so first write
\begin{align*}
	\sum_{k \in \ZZ^N}&\FT{D^s_x \chi_m[ (\chi_m v) (\chi_m w_x)]}(k) \FT D^s_x w(-k)  \\
	&=\sum_{k \in \ZZ^N}\FT{D^s_x [ (\chi_m v)(\chi_m w_x)]}(k) \FT D^s_x \chi_m w(-k) \\
	&= \sum_{k \in \ZZ^N}\FT{D^s_x [ (\chi_m v) (\chi_m w_x)] - (\chi_m v) D^s_x (\chi_m w_x)}(k) \FT D^s_x \chi_m w(-k) \\
	&\quad + \sum_{k \in \ZZ^N}\FT{ (\chi_m v) D^s_x (\chi_m w_x)}(k) \FT D^s_x \chi_m w(-k).
\end{align*}
Application of the fractional Leibniz rule and energy cancellation argument yield that both terms are bounded by
$\| \chi_m v \|_{H^s_{qp}} \| \chi_m w \|_{H^{s}_{qp}} \| D^s_x \chi_m w \|_{L^2_{qp}}.$
By uniform bound, we see that $I$ is bounded by 
\begin{align*}
	\| w \|_{H^{s}_{qp}} \| D^s_x w \|_{L^2_{qp}} &\lesssim (\| D^s_x w_0 \|_{L^2_{qp}} +  \| D^s_x w \|_{L^2_{qp}}) \| D^s_x w \|_{L^2_{qp}} \\
	&\leq \| D^s_x  w_0 \|_{L^2_{qp}}\| D^s_x w \|_{L^2_{qp}} +  \| D^s_x w \|_{L^2_{qp}}^2.
\end{align*}
Since $\| D^s_x w_0 \|_{L^2_{qp}} = o(1)$ for $\delta_1$ large enough, we have 
\[ 
	I \lesssim o(1) \| D^s_x w \|_{L^2_{qp}} +  \| D^s_x w \|_{L^2_{qp}}^2.
\]
	\item the term $II$ is dealt with similarly to $I:$
\begin{align*}
	\sum_{k \in \ZZ^N}&\FT{D^s_x \chi_m[(\chi_m w)(\chi_m u)_x]}(k) \FT D^s_x w(-k) \\
	&= \sum_{k \in \ZZ^N}\FT{D^s_x [(\chi_m w)(\chi_m u)_x]}(k) \FT D^s_x \chi_m w(-k) \\
	&= \sum_{k \in \ZZ^N}\FT{D^s_x [(\chi_m w)(\chi_m u)_x] - \chi_m w D^s_x (\chi_m u)_x}(k) \FT D^s_x \chi_m w(-k) \\
	&\qquad + \sum_{k \in \ZZ^N}\FT{ (\chi_m w)D^s_x (\chi_m u)_x}(k) \FT D^s_x \chi_m w(-k).
\end{align*}
The first sum is bounded by 
\[ 
\| \chi_m w \|_{H^s_{qp}} \| (\chi_m v)_x\|_{H^{s-1}_{qp}} \leq  \| w \|_{H^s_{qp}} \| v \|_{H^{s}_{qp}} \lesssim \| w \|_{H^s_{qp}} \leq \| D^s_{x} w\|_{L^2_{qp}} +  \| D^s_{x} w_0\|_{L^2_{qp}}.
\]
For the second sum, we bound by 
\begin{align*}
	\| \FT{ (\chi_m w)D^s_x (\chi_m u)_x} \|_{\ell^2_k} \| D^s_x \chi_m w\|_{L^2_{qp}} &\lesssim \| \chi_m w \|_{H^{s-1}_{qp}} \| D^s_x (\chi_m u)_x \|_{L^2_{qp}}\| D^s_x \chi_m w\|_{L^2_{qp}} \\
	&\lesssim \delta_1 \| w \|_{H^{s-1}_{qp}} \| D^s_x u \|_{L^2_{qp}} \| D^s_x w\|_{L^2_{qp}}\\
	&\lesssim \delta_1 \| w \|_{H^{s-1}_{qp}} \| D^s_x w\|_{L^2_{qp}}
\end{align*}
where we used Lemma \ref{regData}.
We then use \eqref{lowerInterp} to gain decay from $\delta_1 \| w \|_{H^{s-1}_{qp}}$:
\begin{align*}
	\delta_1 \| w \|_{H^{s-1}_{qp}} &\lesssim \delta_1 \max\{1/n, 1/m\}^{(s-2)/s} + \delta_1 o((\delta_1)^{-1}) + \delta_1 o((\delta_2)^{-1}) \\
	&= \delta_1 n^{(2-s)/s} + \delta_1 o((\delta_2)^{-1}) +  o(1),
\end{align*}
for $\delta_1$ large enough.
By assumption, $\delta_1 < \delta_2,$ so that 
\[
	\delta_1 \| w \|_{H^{s-1}_{qp}} \lesssim \delta_1 n^{(2-s)/s}  + o(1).
\]
Therefore, $II$ is bounded by 
\[
	(\delta_1 n^{(2-s)/s} + o(1))\| D^s_x w\|_{L^2_{qp}}
\]
Using \eqref{cond1} we have 
\[
	\delta_1 n^{(2-s)/s} \leq n^{(s-2)/s - \epsilon} n^{(2-s)/s} = n^{-\epsilon},
\]
so that $II$ indeed decays:
\[
	(n^{-\epsilon} + o(1)) \| D^s_x w\|_{L^2_{qp}}.
\]
	\item for the term $III$ as before we write 
\begin{align*}
	\sum_{k \in \ZZ^N} &\FT{D^s_x \chi_m[(\chi_m u) ((\chi_n - \chi_m) u)_x]}(k) \FT D^s_x w(-k) \\
	&= \sum_{k \in \ZZ^N} \FT{D^s_x [(\chi_m u) ((\chi_n - \chi_m) u)_x]}(k) \FT D^s_x\chi_m w(-k)\\
		&= \sum_{k \in \ZZ^N} \FT{D^s_x [(\chi_m u) ((\chi_n - \chi_m) u)_x] -(\chi_m u) D^s_x [ ((\chi_n - \chi_m) u)_x]}(k) \FT D^s_x \chi_m w(-k)\\
		&\qquad + \sum_{k \in \ZZ^N} \FT{ \chi_m u  D^s_x (\chi_n - \chi_m) u_x}(k) \FT D^s_x \chi_m w(-k).
\end{align*}
The first sum is bounded by
\[
\| \chi_m u \|_{H^s_{qp}} \| (\chi_n - \chi_m) u)_x \|_{H^{s-1}_{qp}} \| D^s_x \chi_m w\|_{L^2_{qp}} \lesssim n^{-l_2} (\delta_1)^{l_2}\| D^s_x \chi_m w\|_{L^2_{qp}},
\]
where we invoked Lemma \ref{DifferenceEst} and Lemma \ref{refinedbound}. 
For the second sum, we bound similarly:
\begin{align*}
	\| \FT{ \chi_m u  D^s_x (\chi_n - \chi_m) u_x} \|_{\ell^2_k} \| D^s_x w \|_{L^2_{qp}} &\lesssim \| \FT{ \chi_m u} \|_{\ell^1_k} \| \FT{D^s_x (\chi_n - \chi_m) u_x} \|_{\ell^2_k}\| D^s_x w \|_{L^2_{qp}} \\
	&\lesssim \max\{1/n, 1/m\}^{l_3} (\delta_1)^{l_3+1},
\end{align*}
for $l_3>0$.
Letting $l_2 = l_3 = 1$, we find the term $IV$ is bounded above by 
\[ \Big( n^{-1} \delta_1 + n^{-1} \delta_1^{2} \Big)\| D^s_x w\|_{L^2_{qp}}. \]
\item the term $IV$ is very similar to $III.$ Decomposing $IV$ as 
\begin{align*}
	\sum_{k \in \ZZ^N}&\FT{D^s_x \chi_m[ ((\chi_n - \chi_m)u) (\chi_n u)_x]}(k) \FT D^s_x w(-k) \\
	&= \sum_{k \in \ZZ^N}\FT{D^s_x [ ((\chi_n - \chi_m)u) (\chi_n u)_x]}(k) \FT D^s_x \chi_m w(-k) \\
	&= \sum_{k \in \ZZ^N}\FT{D^s_x [ ((\chi_n - \chi_m)u) (\chi_n u)_x] -  ((\chi_n - \chi_m)u) D^s_x (\chi_m w)_x }(k) \FT D^s_x \chi_m w(-k) \\
	&\qquad + \sum_{k \in \ZZ^N}\FT{ ((\chi_n - \chi_m)u) D^s_x (\chi_n u)_x }(k) \FT D^s_x \chi_m w(-k),
\end{align*}
we find that it is bounded by
\[ \Big( \delta_1 \max\{1/n, 1/m\}^{l_4} +\max\{1/n, 1/m\}^{l_5} (\delta_1)^{l_5} \Big)\| D^s_x w\|_{L^2_{qp}}, \]
with $l_4 \leq 1$.
Again, let $l_4= l_5 = 1$ so that $V$ has the bound
\[ 
n^{-1} \delta_2 \| D^s_x w\|_{L^2_{qp}}.
\]
\item for the term $V$, first apply Cauchy-Schwarz inequality and Lemma \ref{DifferenceEst} to obtain
\[ 
	\max\{1/n, 1/m\}^{l_6} \| D^{s+l_6} (\chi_n u \chi_n u_x) \|_{L^2_{qp}} \| D^s_x w \|_{L^2_{qp}}.
\]
This is bounded by
\[
	\max\{1/n, 1/m\}^{l_6} \| \chi_n u \chi_n u_x \|_{H^{s+l_6}_{qp}} \| D^s_x w \|_{L^2_{qp}}.
\]
Note that by applying the algebra property, Lemma \ref{refinedbound}, and the uniform bound, we get
\[
	\| \chi_n u \chi_n u_x \|_{H^{s+l_6}_{qp}} \leq \| \chi_n u\|_{H^{s+l_6}_{qp}} \| \chi_n  u_x \|_{H^{s+l_6+ 1}_{qp}} \lesssim \delta_1^{2 l_6 + 1} 
\]
Finally, let $l_6 = 1$, so that the bound is 
\[ 
I \lesssim n^{-1} \delta_1^{3} \| D^s_x w \|_{L^2_{qp}},
\]
as needed.
\end{enumerate}
Combining the estimates for nonlinear and linear terms, we obtain the following bound:
\begin{align*}
	&\Big( n^{-1} \delta_1 + n^{-1} \delta_1^{2} + n^{-1} \delta_1^{3} \Big) \| D^s_x w\|_{L^2_{qp}} + (n^{-\epsilon} + o(1)) \| D^s_x w\|_{L^2_{qp}} + \| D^s_x w \|_{L^2_{qp}}^{2}.
\end{align*}
Note that since $\delta_1$ is big enough, it is enough to estimate $n^{-1} \delta_1^{3}$. 

Using \eqref{cond1} and \eqref{cond2}, we obtain
\begin{align*}
	n^{-1}\delta_1^3 &\leq n^{-2/s - \epsilon} \delta_1^{2} \\
	&\leq n^{-2/s - \epsilon} n^{2((s-2)/s - \epsilon)}.
\end{align*}
Combining the exponents yields the power
\begin{align*}
	-2/s - \epsilon + 2((s-2)/s - \epsilon) &= -2/s + 2 (s-2)/s - 3\epsilon \\
	&= (-2 + 2s - 4)/s - 3\epsilon \\
	&= 2(s-3)/s- 3\epsilon.
\end{align*}
Here, note that $\frac{2(s-3)}{s} <  \frac{2(s-2)}{s},$
so that
\[
	2(s-3)/s- 3\epsilon < \frac{2(s-2)}{s} - 3 \epsilon < 0
\]
by choice of $\epsilon$.
Thus, we have
\[ 
	n^{-1}\delta_1^3 \| D^s_x w \|_{L^2_{qp}} \leq n^{\frac{2(s-2)}{s} - 3 \epsilon} \| D^s_x w \|_{L^2_{qp}}.
\]
Therefore, we obtain 
\begin{align*}
	\frac{\D}{\D t} \| D^s_x w \|_{L^2_{qp}}^2 \lesssim \Big( n^{\frac{2(s-2)}{s} - 3 \epsilon} + n^{-\epsilon} + o(1) \Big) \| D^s_x w \|_{L^2_{qp}} + \| D^s_x w \|_{L^2_{qp}}^2,
\end{align*}
so that 
\[
	\frac{\D}{\D t} \| D^s_x w \|_{L^2_{qp}} \lesssim \Big( n^{\frac{2(s-2)}{s} - 3 \epsilon} + n^{-\epsilon} + o(1) \Big) + \| D^s_x w \|_{L^2_{qp}}.
\]
Finally, a Gronwall type argument reveals that 
\begin{align*}
	\| D^s_x w \|_{L^2_{qp}} &\lesssim e^t ( n^{\frac{2(s-2)}{s} - 3 \epsilon} + n^{-\epsilon} + o(1) ) + \| D^s_x w_0 \|_{L^2_{qp}} \\
	&\lesssim n^{\frac{2(s-2)}{s} - 3 \epsilon} + n^{-\epsilon} + o(1),
\end{align*} 
since $\| D^s_x w_0 \|_{L^2_{qp}} = o(1)$ for $\delta_1$ large enough.
Furthermore, since $\delta_2 < n^{(s-2)/s - \epsilon},$ we know that 
\[n^{\frac{2(s-2)}{s} - 3 \epsilon} + n^{-\epsilon} \leq \delta_1^{-l}\] 
for some $l > 0$. 
Thus, we obtain $\| D^s_x w \|_{L^2_{qp}} = o(1).$
The same analysis applied to $\| w \|_{L^2_{qp}}$ yields $\| w \|_{L^2_{qp}} = o(1)$, 
so that we obtain the desired result.
\end{proof}

Note that now that we have the Cauchy bound, we can also perform Cauchy-like estimates on the linear and nonlinear terms. 
This will be useful in the next section.
\begin{cor}\label{CauchyLikeBound}
Suppose the same assumptions hold as in Proposition \ref{HsCauchyEstimate}. 
First, the following bound holds on the difference of nonlinearities:
\begin{equation}\label{nonlinearCauchyBound}
	\| \chi_n[ (\chi_n u) (\chi_n u_x)] - \chi_m[ (\chi_m v) (\chi_m v_x)] \|_{H^{s-1 -  \gamma_1}_{qp}} \lesssim o(1) + \max\{1/n, 1/m\}^{\gamma_1},
\end{equation}
where $\gamma_1>0$ is small enough so that $s - 1 - \gamma_1 > N/2.$
Second, the following bound on holds on the difference of linear terms:
\begin{equation}\label{linearCauchyBound}
	\| \chi_n Hu_{xx} - \chi_m Hv_{xx} \|_{H^{s - 2 \gamma_2}_{qp}} \lesssim o(1) + \max\{1/n, 1/m\}^{\gamma_2},
\end{equation}
where $\gamma_2 > 0$ is small enough.
Furthermore, both \eqref{nonlinearCauchyBound} and \eqref{linearCauchyBound} hold uniformly on $[0,T]$.
\end{cor}

\begin{proof}
We only sketch the proof. 
For the difference of nonlinear terms, use the decomposition \eqref{nonlinearterm} to obtain the following:
\begin{align*}
	&\| \chi_n[ (\chi_n u) (\chi_n u_x)] - \chi_m[ (\chi_m v) (\chi_m v_x)] \|_{H^{s-1 -  \gamma_1}_{qp}} \\
	&\quad \leq \| \chi_m[ (\chi_m v) (\chi_m w_x)]\|_{H^{s-1 -  \gamma_1}_{qp}} + \|\chi_m[(\chi_m w)(\chi_m u)_x] \|_{H^{s-1 -  \gamma_1}_{qp}} \\
	&\quad+ \|\chi_m[(\chi_m u) ((\chi_n - \chi_m) u)_x]\|_{H^{s-1 -  \gamma_1}_{qp}} + \| \chi_m[ ((\chi_n - \chi_m)u) (\chi_n u)_x]\|_{H^{s-1 -  \gamma_1}_{qp}} \\
	&\quad + \|[\chi_n - \chi_m][\chi_n u \chi_n u_x] \|_{H^{s-1 -  \gamma_1}_{qp}}.
\end{align*}
As $s - 1 -\gamma_1 > N/2$, we can apply the algebra property on the first four terms.
For the fifth term, first apply Lemma \ref{DifferenceEst} and then the algebra property.
These operations altogether let us bound the above by
\begin{align*}
	&\| \chi_n[ (\chi_n u) (\chi_n u_x)] - \chi_m[ (\chi_m v) (\chi_m v_x)] \|_{H^{s-1 -  \gamma_1}_{qp}}\\
	&\quad \leq \| v \|_{H^{s-1 -  \gamma_1}_{qp}} \|w_x\|_{H^{s-1 -  \gamma_1}_{qp}} 
	+ \| w \|_{H^{s-1 -  \gamma_1}_{qp}} \|u_x \|_{H^{s-1 -  \gamma_1}_{qp}} \\
	&\qquad+ \| u \|_{H^{s-1 -  \gamma_1}_{qp}}\| (\chi_n-\chi_m)u_x \|_{H^{s-1 -  \gamma_1}_{qp}}
	+ \| (\chi_n - \chi_m)u\|_{H^{s-1 -  \gamma_1}_{qp}} \|u_x \|_{H^{s-1 -  \gamma_1}_{qp}} \\
	&\qquad+ \max\{ 1/n, 1/m\}^{\gamma_1} \| u \|_{H^{s-1 -  \gamma_1}_{qp}} \| u_x\|_{H^{s-1 -  \gamma_1}_{qp}}.
\end{align*}
Simplifying and applying Lemma \ref{DifferenceEst} on the third and fourth term yields the bound
\begin{align*}
	&\| \chi_n[ (\chi_n u) (\chi_n u_x)] - \chi_m[ (\chi_m v) (\chi_m v_x)] \|_{H^{s-1 -  \gamma_1}_{qp}}\\
	&\qquad \leq \| v \|_{H^{s}_{qp}} \|w\|_{H^{s}_{qp}} 
	+ \| w \|_{H^{s}_{qp}} \|u \|_{H^{s}_{qp}} 
	+ \max\{ 1/n, 1/m\}^{\gamma_1} \| u \|_{H^{s}_{qp}}^2.
\end{align*}
Applying the uniform bound and the Cauchy estimate yields the bound
\[ 
	\| \chi_n[ (\chi_n u) (\chi_n u_x)] - \chi_m[ (\chi_m v) (\chi_m v_x)] \|_{H^{s-1 -  \gamma_1}_{qp}} \lesssim o(1) + \max\{ 1/n, 1/m\}^{\gamma_1}.
\]
For the difference of linear terms, similarly use the decomposition \eqref{linearterm}.
As such, applying Lemma \ref{DifferenceEst}, we obtain:
\begin{align*}
	\| \chi_n Hu_{xx} - \chi_m Hv_{xx} \|_{H^{s - 2 -\gamma_2}_{qp}} &\leq \| (\chi_n - \chi_m)u_{xx} \|_{H^{s - 2 \gamma_2}_{qp}} + \| \chi_m Hw_{xx} \|_{H^{s - 2 - \gamma_2}_{qp}}  \\
	&\lesssim \max\{ 1/n, 1/m\}^{\gamma_2} \| u_{xx} \|_{H^{s-2}_{qp}} + \| w \|_{H^{s-\gamma_2}_{qp}}  \\
	&\lesssim \max\{ 1/n, 1/m\}^{\gamma_2} \| u \|_{H^s_{qp}} + \| w \|_{H^{s}_{qp}}.
\end{align*}
Clearly the latter can be bounded by 
\[ \| \chi_n[ (\chi_n u) (\chi_n u_x)] - \chi_m[ (\chi_m v) (\chi_m v_x)] \|_{H^{s-1 -  \gamma_1}_{qp}}\leq \max\{ 1/n, 1/m\}^{\gamma_2} + o(1),\]
so that we obtain the result.
\end{proof}

Finally, we perform the following continuity in the initial data for the regularized Benjamin-Ono equation.
Namely, let $u_0, v_0 \in \mathcal{K} \subset H^s_{qp},$ where $\mathcal{K}$ is compact in $H^s_{qp}$. 
Let $(u_0)_\delta$ be a regularized version of $u_0$ as defined in \eqref{regData}, with parameter $\delta>0$.
Let $n,m \in \NN$. Let $u_{n}$ solve
\begin{align*}
	u_t &= \chi_n[(\chi_n u_n) (\chi_n u_{nx})] + \chi_n[H u_{nxx}],  \\
	u(0,\cdot) &= (u_0)_{\delta},
\end{align*}
whereas $u_{m}$ denotes the solution of
\begin{align*}
	u_t &= \chi_m[(\chi_m u_m) (\chi_m u_{mx})] + \chi_m[H u_{mxx}], \\
	u(0,\cdot) &= (v_0)_{\delta}.
\end{align*}
For ease of notation, write $u = u_{n}, v = u_{m}, w = u-v.$ 
Since $\mathcal{K}$ is compact, and so by Corollary \ref{UniformBound}, $u$ and $v$ exist on the same time interval.

We have the following result:
\begin{prop}\label{HsContData}
Let $m, n \in \RR_{+}$ and $\delta \in \RR_{+}$. 
Assume $m \geq n$. 
Let $u_0, v_0 \in \mathcal{K} \subset H^s_{qp},$ where $\mathcal{K}$ is compact in $H^s_{qp}$. 
Furthermore, assume $\delta \leq n^{(s-2)/s - \epsilon},$ for some $\epsilon>0.$ 
Let $u_{n}, u_{m}$ be solutions for the respective regularized Benjamin-Ono problem which exist on the same time interval $[0,T]$, defined as above.

Given $\eta > 0$, we can pick $\delta>0$ sufficiently large so that there is $\tilde{\eta}>0$ such that
$\| u_0 - v_0\| < \tilde{\eta}$ implies 
\[
\| u_{n} - u_{m} \|_{L^\infty_T H^s_{qp}} < \eta.
\]
In addition, the implicit constants depend only on the norm of the initial data and as such can be taken uniform as long as the initial data comes from a compact set.
\end{prop}

\begin{proof}
The argument is almost identical to the proof of Proposition \ref{HsCauchyEstimate}.
The $L^2_{qp}$ estimate becomes
\[ 
	\| w \|_{L^2_{qp}} \lesssim \max\{1/n, 1/m\}^{s-2} + \| u_0 - v_0 \|_{L^2_{qp}}.
\]
The \eqref{lowerInterp} becomes
\begin{align*}
	\| w\|_{H^{s-1}_{qp}} &\lesssim \Big( \max\{1/n, 1/m\}^{s-2} + \|(u_0)_{\delta} - (v_0)_{\delta} \|_{L^2_{qp}} \Big)^{1/s} \\
	&\leq \max\{1/n, 1/m\}^{(s-2)/s} + \|(u_0)_{\delta} - (v_0)_{\delta} \|_{L^2_{qp}}^{1/s} \\
	&\leq \max\{1/n, 1/m\}^{(s-2)/s} + \| u_0 - v_0 \|_{L^2_{qp}}^{1/s}.
\end{align*}
As for the $D^s_x$ estimate, the only place where the argument changes is in the estimate of $II$.
Here, we have
\begin{align*}
	\sum_{k \in \ZZ^N}&\FT{D^s_x \chi_m[(\chi_m w)(\chi_m u)_x]}(k) \FT D^s_x w(-k) \\
	&= \sum_{k \in \ZZ^N}\FT{D^s_x [(\chi_m w)(\chi_m u)_x]}(k) \FT D^s_x \chi_m w(-k) \\
	&= \sum_{k \in \ZZ^N}\FT{D^s_x [(\chi_m w)(\chi_m u)_x] - \chi_m w D^s_x (\chi_m u)_x}(k) \FT D^s_x \chi_m w(-k) \\
	&\qquad + \sum_{k \in \ZZ^N}\FT{ (\chi_m w) D^s_x (\chi_m u)_x}(k) \FT D^s_x \chi_m w(-k).
\end{align*}
The first sum is bounded by 
\[ 
\| \chi_m w \|_{H^s_{qp}} \| (\chi_m v)_x\|_{H^{s-1}_{qp}} \leq  \| w \|_{H^s_{qp}} \| v \|_{H^{s}_{qp}} \lesssim \| w \|_{H^s_{qp}} \leq \| D^s_{x} w\|_{L^2_{qp}} +  \| D^s_{x} w_0\|_{L^2_{qp}}.
\]
For the second sum, we bound by 
\begin{align*}
	\| \FT{ (\chi_m w)D^s_x (\chi_m u)_x} \|_{\ell^2_k} \| D^s_x \chi_m w\|_{L^2_{qp}} &\lesssim \| \chi_m w \|_{H^{s-1}_{qp}} \| D^s_x (\chi_m u)_x \|_{L^2_{qp}}\| D^s_x \chi_m w\|_{L^2_{qp}} \\
	&\lesssim \delta \| w \|_{H^{s-1}_{qp}} \| D^s_x u \|_{L^2_{qp}} \| D^s_x w\|_{L^2_{qp}}\\
	&\lesssim \delta \| w \|_{H^{s-1}_{qp}} \| D^s_x w\|_{L^2_{qp}}
\end{align*}
where we used Lemma \ref{regData}.
We then use \eqref{lowerInterp} to gain decay from $\delta \| w \|_{H^{s-1}_{qp}}$:
\begin{align*}
	\delta \| w \|_{H^{s-1}_{qp}} &\lesssim \delta \max\{1/n, 1/m\}^{(s-2)/s} + \delta \| u_0 - v_0 \|_{L^2_{qp}}^{1/s} \\
	&= \delta n^{(2-s)/s} + \delta \| u_0 - v_0 \|_{L^2_{qp}}^{1/s},
\end{align*}
for $\delta_1$ large enough.
Therefore, $II$ is bounded by 
\[
	(\delta n^{(2-s)/s} + \delta \| u_0 - v_0 \|_{L^2_{qp}}^{1/s} )\| D^s_x w\|_{L^2_{qp}}.
\]
Combining all the estimates as in the proof of Proposition \ref{HsCauchyEstimate} we have
\[ 
	\| D^s_x w \|_{L^2_{qp}} \lesssim n^{\frac{2(s-2)}{s} - 3 \epsilon} + n^{-\epsilon} + \delta \| D^s_x(u_0 - v_0) \|_{L^2_{qp}}^{1/s}.
\]
Since $\frac{2(s-2)}{s} - 3 \epsilon < 0$, there is some $\ell > 0$ such that 
\[ n^{\frac{2(s-2)}{s} - 3 \epsilon} + n^{-\epsilon} \leq \delta^{-\ell}. 
\]
Combining this with the $L^2_{qp}$ estimate we have
\[ 
\| w \|_{H^s_{qp}} \leq \| w \|_{L^2_{qp}} + \| D^s_x w \|_{L^2_{qp}} \lesssim n^{2-s} +  n^{\frac{2(s-2)}{s} - 3 \epsilon} + n^{-\epsilon} + \delta \| D^s_x(u_0 - v_0) \|_{L^2_{qp}}^{1/s} + \| u_0 - v_0 \|_{L^2_{qp}}. 
\]
Since all the exponents of $n$ are negative, there is some $\ell> 0$ such that 
\[
n^{2-s} +  n^{\frac{2(s-2)}{s} - 3 \epsilon} + n^{-\epsilon} \leq \delta^{-\ell},
\]
so that 
\[ 
\| w \|_{H^s_{qp}} \lesssim \delta^{-\ell} + \delta \| D^s_x(u_0 - v_0) \|_{L^2_{qp}}^{1/s} + \| u_0 - v_0 \|_{L^2_{qp}}, 
\]
where this inequality holds up to some constant $A$.

Given $\eta > 0$, pick $\delta>0$ big enough so that $\delta^{-\ell} + \delta^{-1} +\delta^{-2s} < A^{-1}\eta$.
Now, pick $\tilde{\eta} = \delta^{-2s}$. 
Therefore, 
\[ 
	\| w \|_{H^s_{qp}} \leq A(n^{-\ell} + \delta (\delta^{-2s})^{1/s} + \delta^{-2s}) = A(\delta^{-\ell} + \delta^{-1} + \delta^{-2s}) < \eta,
\]
which completes the proof.
\end{proof}
\section{Local Wellposedness of Benjamin-Ono equation}
In this section, we combine the results obtained earlier to prove local-wellposedness of the Benjamin-Ono equation. 
There are several items to show:
\begin{itemize}
	\item the existence of a solution $u$ to the Benjamin-Ono equation \eqref{LWPeq1} in $H^s_{qp}$,
	\item uniqueness of the solution $u$,
	\item the continuity in time of the norm $\| u \|_{H^s_{qp}}$,
	\item the continuity in the initial data.
\end{itemize}

We begin with the first item.
\begin{prop}
Let $s > N/2 + 1$. 
Let $u_0 \in H^{s}_{qp}$.
Then, there exists $T = T(s, \| u_0 \|_{H^s_{qp}})$ such that there exists a unique solution $u \in L^\infty([0,T]; H^s_{qp})$ of the Benjamin-Ono equation with the initial data $u_0$.
In particular, 
\begin{itemize}
	\item if $s > N/2 + 2$, then $u$ solves \eqref{LWPeq1} pointwise,
	\item if $N/2 + 2 \geq s > N/2 + 1$, then $u$ solves \eqref{LWPeq1} in the sense of distributions.
\end{itemize}
\end{prop}
\begin{proof}
Consider the sequence $\{ (u_0)_\delta \}$ of regularized initial data \eqref{regData}.
Given $\delta>0,$ pick  $n_\delta = \exp(\frac{s}{s-2-\epsilon s} \log \delta),$ where $\epsilon >0$ is chosen as in Proposition \ref{HsCauchyEstimate}.
With this choice, we satisfy $\delta = n^{(s-2)/s - \epsilon}$.

Consider the sequence $u^{\delta} := u_{n_\delta, \delta}$. 
From the previous section, we know that this sequence is Cauchy in $L^\infty([0,T]; H^s_{qp})$, so it has a limit $u$.
Furthermore, from Corollary \ref{UniformBound}, it is known that the time of existence $T$ satisfies the bound
\[
	T < \frac{1}{C \| u_0 \|_{H^s_{qp}}}.
\]
The independence of the bound on $T$ from $\delta$ and $n_\delta$ ensures that when the limit is taken, the time of existence is not affected.
Thus, $u \in L^\infty([0,T]; H^s_{qp})$.

We check that $u$ solves the Benjamin-Ono equation.

Let $\gamma > 0$ be such that $s-2-\gamma >0$ and $s-1-\gamma > N/2$. 
By Corollary \ref{CauchyLikeBound}, we know that 
\[
	\| \chi_n[ (\chi_n u) (\chi_n u_x)] - \chi_m[ (\chi_m v) (\chi_m v_x)] \|_{H^{s-1 -  \gamma}_{qp}} \lesssim o(1) + \max\{1/n, 1/m\}^{\gamma},
\]
and
\[
	\| \chi_n Hu_{xx} - \chi_m Hv_{xx} \|_{H^{s - 2 - \gamma}_{qp}} \lesssim o(1) + \max\{1/n, 1/m\}^{\gamma}.
\]

Letting $n = n_\delta$ and $m, \delta_2 \to \infty$ yields the following bound:
\begin{equation}\label{nonlinearSolBound}
	\| \chi_n[ (\chi_n u^\delta) ( \chi_n u^\delta_x)] - u u_x \|_{H^{s-1 -  \gamma}_{qp}} \lesssim o(1) + n^{-\gamma},
\end{equation}
and 
\begin{equation}\label{linearSolBound}
	\| \chi_n Hu^{\delta}_{xx} - Hu_{xx} \|_{H^{s - 2 \gamma}_{qp}} \lesssim o(1) + n^{-\gamma},
\end{equation}
where for the ease of writing we write $n := n_\delta$.

By integrating, we have
\begin{equation}\label{ModelEq}
	u^\delta (t, x) = (u_0)_\delta(x) + \int^t_0 \chi_n[(\chi_n u^\delta) (\chi_n u^\delta_{x})] + \chi_n[H u^\delta_{xx}] \D s,
\end{equation} 
where $n$ is defined as above.

To show that $u$ solves the original Benjamin-Ono problem, we would like to obtain that 
\begin{equation}\label{LimitEq}
	u(t, x) = u_0(x) + \int^t_0 uu_x(s,x) + H u_{xx}(s,x) \D s
\end{equation}
holds pointwise or in the sense of distributions.

Consider the difference of the right-hand sides of \eqref{ModelEq} and \eqref{LimitEq}:
\begin{equation}\label{IntDifference}
	\int^t_0 \chi_n[(\chi_n u^\delta) (\chi_n u^\delta_{x})] - uu_x(s,x) + \chi_n[H u^\delta_{xx}] - H u_{xx}(s,x) \D s
\end{equation}
Taking the $H^{s - 2 - \gamma}_{qp}$ norm of the integrand and applying the bounds \eqref{nonlinearSolBound} and \eqref{linearSolBound}, we are able to
obtain that this is  of the form $o(1) + n^{-\gamma}.$
Letting $\delta \to \infty$ implies $n \to \infty$ as well. 
Since the bounds are uniform in $t$, we also obtain that the expression \eqref{IntDifference} vanishes as $\delta \to \infty$. 
Since $u^\delta \to u$ and $(u_0)_\delta \to u_0$ in $H^{s}_{qp}$, they also converge in $H^{s-2}_{qp}$.
Therefore, \eqref{LimitEq} holds in $H^{s-2}_{qp}$.

Finally, if $s > N/2 + 2,$ then $s-2 > N/2$. By the Sobolev inequality, the equation \eqref{LimitEq} holds pointwise.
If $N/2 + 2 \geq s > N/2 + 1$, then $N/2 \geq s > N/2 - 1$, and so the equation \eqref{LimitEq} holds in the sense of distributions.
\end{proof}

\begin{prop}[Weak continuity in time]
The limit $u$ is weakly continuous in time in $H^{s}_{qp}.$
\end{prop}
\begin{proof}
Let $\phi \in H^{-s}_{qp}$.
Note
\begin{align*}
	\langle \phi, u(s) - u(t) \rangle = \langle \phi, u(s) - u^\delta(s) \rangle + \langle \phi, u^\delta(s) - u^\delta(t) \rangle + \langle \phi, u^\delta(t) - u(t) \rangle.
\end{align*}
By Theorem \ref{LocalExistenceBad}, the regularized solutions are continuous in time. 
As such, the middle term can be made as small as we would like.
As for the first and the third, it is enough to prove that one of them decays.  

Let $\epsilon >0$. For any $m$ satisfying $1 \leq m <s,$ let $\tilde{\phi} \in H^{-m}$ be given such that 
\[ 
\| \phi - \tilde{\phi} \|_{H^{-s}_{qp}} \leq \frac{\epsilon}{3}.
\]
Now, write and bound
\begin{align*}
	\langle \phi, u(s) - u^\delta(s) \rangle 
	&= \langle \phi- \tilde{\phi}, u(s) \rangle + \langle \tilde{\phi}, u(s) - u^\delta(s)\rangle + \langle \tilde{\phi} - \phi, u^\delta (s) \rangle  \\
	&\leq \| \phi - \tilde{\phi} \|_{H^{-s}_{qp}} \| u \|_{H^{s}_{qp}} + \| \tilde{\phi} \|_{H^{-m}_{qp}} \| u - u^\delta \|_{H^{m}_{qp}}
	+ \| \phi - \tilde{\phi} \|_{H^{-s}_{qp}} \| u^\delta \|_{H^{s}_{qp}}.
\end{align*}
Note that $\| u \|_{H^{s}_{qp}}$ and $\| u^\delta \|_{H^{s}_{qp}}$ are bounded due to uniform bounds. 
Thus, we obtain that
\[
	\langle \phi, u(s) - u^\delta(s) \rangle \lesssim \frac{\epsilon}{3} + \frac{\epsilon}{3} + \frac{\epsilon}{3} = \epsilon.
\]
Therefore, we have that $u \in C_W([0,T]; H^s_{qp})$.
\end{proof}

We move on to showing uniqueness.
\begin{prop}[Uniqueness]
	Let $s > N/2 + 1$. Then, solutions of the Benjamin-Ono equation are unique. 
\end{prop}
\begin{proof}
Let $s > N/2 + 1$ and let $u_0 \in H^{s}_{qp}$. 
Suppose $u,v$ are two solutions of the Benjamin-Ono equation with initial data $u_0$.
Write $w = u-v$. Then, $w$ satisfies
\[ 
	w_t = wu_x + v w_x + Hw_{xx}
\]
and $w(0) = 0$.
Now, as in the proof of the Cauchy estimate one finds that
\[
	\frac{\D \|w\|_{L^2_{qp}}^2}{\D t} \lesssim (\| u \|_{H^s_{qp}} + \| v \|_{H^s_{qp}}) \|w\|_{H^s_{qp}}^2 \leq C\|w\|_{L^2_{qp}}^2,
\]
where $C$ does not depend on $w$ but may depend on the uniform bound of $u$ and $v$.
An application of Gronwall's inequality yields
\[
	\|w(t)\|_{L^2_{qp}}^2 \leq \|w(0)\|_{L^2_{qp}}^2 \exp(Ct) = 0,
\]
since clearly $\|w(0)\|_{L^2_{qp}}^2 = 0.$
Thus, we have $u=v$ almost everywhere.
Now, since $s > N/2 + 1$ and $u,v \in H^{s}_{qp}$, $u$ and $v$ are both continuous. 
Thus, $u=v$ everywhere and we have uniqueness.
\end{proof}

We now are able to prove that $u \in C([0,T]; H^s_{qp})$.  Since $H^{s}_{qp}$ is a Hilbert space and we have already demonstrated weak continuity, 
all that remains to show is continuity of the norm.
\begin{prop} 
For $s > N/2 +1$, we have $\| u(t) \|_{H_{qp}^s}$ is a continuous function of time.
\end{prop}
\begin{proof}
As in the proof of the uniform estimate, we have
\begin{equation*}
	\| u^\delta \|_{H^s_{qp}} \lesssim \| u_0 \|_{H^s_{qp}} + ((\| u_0 \|_{H^s_{qp}})^{-1} - Ct)^{-1}.
\end{equation*}
Letting $\delta \to \infty$ we have
\begin{equation*}
	\| u \|_{H^s_{qp}} \lesssim \| u_0 \|_{H^s_{qp}} + ((\| u_0 \|_{H^s_{qp}})^{-1} - Ct)^{-1}.
\end{equation*}
Now, let $t_n \to 0^+$. The apriori estimate gives  
\begin{equation*}
	\| u(t_n) \|_{H^s_{qp}} \lesssim \| u_0 \|_{H^s_{qp}} + ((\| u_0 \|_{H^s_{qp}})^{-1} - Ct_n)^{-1}.
\end{equation*}
Taking the limit supremum $n \to \infty$, we obtain
\begin{equation*}
	\limsup_{n \to \infty}\| u(t_n) \|_{H^s_{qp}} \lesssim \| u_0 \|_{H^s_{qp}} + \limsup_{n \to \infty} ((\| u_0 \|_{H^s_{qp}})^{-1} - Ct_n)^{-1} =\| u_0 \|_{H^s_{qp}} + \| u_0 \|_{H^s_{qp}}.
\end{equation*}
Now, weak continuity of $u$ in $H^s$ implies that $u(t_n) \rightharpoonup u(0).$ 
By lower semi-continuity of the norm we have 
\[\| u_0 \|_{H^s_{qp}} \leq \liminf_{n \to \infty} \| u(t_n) \|_{H^s_{qp}}.\]
Thus, we obtain
\[ 
	\| u_0 \|_{H^s_{qp}} = \lim_{n \to \infty} \| u(t_n) \|_{H^s_{qp}}
\]
This shows right-continuity at $t = 0^+$.

For any $T^* \in (0,T)$, we interpret $T^*$ as a new initial time. 
We can repeat our arguments for existence and uniqueness of solutions to the regularized Benjamin-Ono equation starting at a new time $T^*.$
We find solutions of the original Benjamin-Ono equation on some time interval around $T^*$. 
By the above argument, these solutions are right-continuous at $t = T^*$. 
By uniqueness, solutions starting at $t = T^{*}$ and solutions starting at $t = 0$ must be the same. 
This is true for any $T^* \in (0,T)$ so the solutions starting at $t=0$ must be right continuous on $(0,T)$.

For left continuity, note that the Benjamin-Ono equation is invariant under the change of variables $(t,x) \mapsto (-t,-x)$. 
From this one obtains left continuity on $(0,T)$.
We conclude that the norm $\|u(t) \|_{H^s_{qp}}$ is continuous in time.
\end{proof}

\begin{rmk}\label{ContinuityInInitialDataForRegEquation}
Let $\eta > 0$ be given. 
By the Cauchy estimate in Proposition \eqref{HsCauchyEstimate}, there exists $N > 0$ such that for all $\delta > N,$ uniformly in $u_0 \in \mathcal{K},$ we have
\[ \sup_{t \in [0,T] } \| u^{\delta} -u \|_{H^s_{qp}} < \eta, \]
where $u^{\delta}$ and $u$ solve the regularized and the original Benjamin-Ono equation. 
The uniformity in $u_0$ comes from the fact the rate of decay is uniform for $\mathcal{K}$ compact.
\end{rmk}

\begin{prop}[Continuity in initial data in $H^{s}_{qp}$] 
Let $s > N/2 + 1.$ 
Let $u_0 \in H^s_{qp}$ and suppose $\{ u_n \}_{n=1}^{\infty}$ is a sequence such that $u_n \to u_0$ in $H^s_{qp}.$
Let $T$ be the common time of existence for solutions of Benjamin-Ono equation with initial data $\{ u_n \}_{n=0}^{\infty}$, which can be taken independent of $n$.
Let $U_n \in C([0,T]; H^s_{qp})$ be the solution of the Benjamin-Ono equation with initial data $u_n$ for $n= 0, 1, 2, \ldots$.
Then,
\[ \lim_{n\to \infty} \sup_{t \in [0,T]} \| U_n - U_0 \|_{H^s_{qp}} = 0.\]
\end{prop}
\begin{proof}
Let $\eta > 0$. Given any $\delta > 1$, write 
\[
	U_n - U_0 = (U_n - (U_n)^\delta) + ((U_n)^\delta - (U_0)^\delta) + ((U_n)^\delta - U_0),
\]
where $(U_n)^\delta$ is the solution of the regularized Benjamin-Ono equation with initial data $u_n$ and the regularizing parameter $\delta$.

It is clear that the set $\{ u_n \}_{n=1}^{\infty} \cup u_0$ is a compact subset of $H^s_{qp}$.
Therefore, per Remark \ref{ContinuityInInitialDataForRegEquation}, we can pick $\delta>0$ such that for all $n = 0, 1, 2, \ldots$, 
we have 
\[ 
	\sup_{t \in [0,T] } \| (U_n)^\delta - U_n\|_{H^s_{qp}} < \frac{\eta}{3}.
\]
It remains to deal with the difference $(U_n)^\delta - (U_0)^\delta$. 
This follows from Proposition \ref{HsContData}: we can pick $\delta$ sufficiently big so that  
\[ \| (U_n)^\delta - (U_0)^\delta \|_{H^s_{qp}} \leq \frac{\eta}{3}. \]
This completes the proof.
\end{proof}

\section{Conservation Laws}
In this section, we explain the challenges of extending local time of existence. 
In the case of periodic or decaying data, this is usually solved by means of conserved quantities.
These quantities, also known as conservation laws, allow to control the Sobolev norms when data is periodic or decaying.
Once local solutions are obtained, these quantities allow one to conclude that the Sobolev norms are bounded on any time interval.
In return, this allows to conclude that the local solutions exist on any time interval. 
The first four conservation laws are
\begin{itemize}
	\item mass: 
\begin{equation}
	 \int u \D x.
\end{equation}
\item momentum: 
\begin{equation}
	\int u^2 \D x,
\end{equation}
\item energy: 
\begin{equation}\label{CQenergy}
	\int \frac{u^3}{3} + u Hu_x \D x, 
\end{equation}
\item the conservation law for the $H^1$ norm:
\begin{equation}
	\int \frac{u^4}{4} + \frac{3}{2} u^2 Hu_x - \frac{3}{2} u_x^2 \D x.
\end{equation}
\end{itemize}

When adapted to the case of quasiperiodic data, the conserved quantities still hold.
However, we show below that conservation laws no longer seem to control the quasiperiodic Sobolev norms.
For the reader's convenience, we illustrate how the conservation laws are used by considering the energy \eqref{CQenergy} of Benjamin-Ono equation, with decaying initial data.

Since \eqref{CQenergy} is conserved, it is equal to some number, say $\alpha$, that only depends on $\|u_0\|_{H^{1/2}}$ and not $t$.
Now, note that the nonhomogenous part of the $H^{1/2}$ norm for the decaying data can be written as $\| u \|^2_{\dot{H}^{1/2}} =  \int u Hu_x \D x$.
Thus, $\| u \|^2_{\dot{H}^{1/2}} = \alpha - \int \frac{u^3}{3} \D x$. 
Therefore we have 
\[  
	\| u \|^2_{\dot{H}^{1/2}} \lesssim |\alpha| + \frac{1}{3}\|u \|_{L^3}^3.
\]
By Gagliardo-Nirenberg inequality, we know that $\|u \|_{L^3}^3$ is bounded by $\|u \|_{L^2}^2 \|\partial^{1/2}_x u \|_{L^2},$ i.e. $\|u \|_{L^3} \leq C \|u \|_{L^2}^2 \|\partial^{1/2}_x u \|_{L^2}$ for a fixed $C>0$. 
Applying Cauchy-Schwarz inequality then yields that 
\[
	\|u \|_{L^3}^3 \leq 2C^4 \|u \|_{L^2}^4 + \frac{1}{2}\|\partial^{1/2}_x u \|_{L^2}^2.
\]
We then have that 
\[
	\frac{1}{2}\| u \|^2_{\dot{H}^{1/2}} \leq |\alpha| + 2C^4 \|u \|_{L^2}^4 = |\alpha| + 2C^4 \|u_0 \|_{L^2}^4,
\]
where we used the conservation of $L^2$ norm. 
From here it follows that $\| u \|_{H^{1/2}}$ is bounded only by the norm of the initial data, and we obtain the desired control.

Now, we consider the conserved quantities in the case of quasiperiodic data. 
These conservation laws are listed below:
\begin{itemize}
	\item mass: 
\begin{equation}\label{CQ1}
	\quad \lim_{R \to \infty}\frac{1}{2R}\int^R_{-R} u \D x \text{ or } \hat{u}(0),
\end{equation}
\item momentum: 
\begin{equation}\label{CQ2}
	\quad \lim_{R \to \infty}\frac{1}{2R}\int^R_{-R} u^2 \D x \text{ or } \sum_{k \in \ZZ^N} \hat{u}(k) \hat{u}(-k),
\end{equation}
\item energy: 
\begin{equation}\label{CQ3} 
	\lim_{R \to \infty}\frac{1}{2R}\int^R_{-R} \frac{u^3}{3} + u Hu_x \D x 
	\text{ or } \sum_{k \in \ZZ^N} \frac{1}{3}\FT{u^2} (k) \hat{u}(-k) + \FT{Hu_x}(k)\hat{u}(-k),
\end{equation}
\item conservation law for the $H^1$ norm:
\begin{equation}\label{CQ4} 
	\begin{split}\lim_{R \to \infty}&\frac{1}{2R}\int^R_{-R} \frac{u^4}{4} + \frac{3}{2} u^2 Hu_x - \frac{3}{2} u_x^2 \D x 
	\\
&\text{or } \sum_{k\in\ZZ^N} \frac{1}{4} \FT{u^3}(k) \hat{u}(-k) + \frac{3}{2} \FT{u^2}(k) \FT{Hu_x}(-k) - \frac{3}{2} \hat{u_x}(k)\hat{u_x}(-k). \end{split}
\end{equation}
\end{itemize}
Before we prove the conservation laws, let us show why the laws do not seem to control the quasiperiodic Sobolev norm.
A crucial step in the example above is the Gagliardo-Nirenberg inequality. 
It is thus natural to ask if a quasiperiodic version of this inequality holds.
In contrast to decaying and periodic spaces, it appears to false, by a simple scaling argument \cite{JasonZhaoPrivate}.

\begin{rmk}[Failure of the Gagliardo-Nirenberg inequality]
Define a mean value type norm ${\mathcal{L}^p}$ by 
\[
\| u \|_{\mathcal{L}^p}^p = \lim_{R \to \infty} \frac{1}{2R} \int_{-R}^R |u|^p \D x.
\]
These norms are invariant under rescaling $u_{\lambda}(x) = u(x/\lambda),$ i.e. $\| u \|_{\mathcal{L}^p} = \| u_\lambda \|_{\mathcal{L}^p}$ for any $\lambda > 0$.
Furthermore, it is straightforward to check that $\widehat{D^s_x u_\lambda}(k) =\lambda^{-s}\widehat{D^s_x u}(\lambda k)$ and that $D^s_x u_\lambda(x) = \lambda^{-s} D^s_x u (x/\lambda)$ for any $\lambda > 0$.
It then follows that $\| D^s_x u_\lambda \|_{\mathcal{L}^p} = \lambda^{-s} \| D^s_x u \|_{\mathcal{L}^p}.$
As a result, if any inequality of the form 
\[ \| u \|_{\mathcal{L}^q} \lesssim \| u \|_{\mathcal{L}^r}^{1-\theta} \| D^s_x u \|_{\mathcal{L}^p}^{\theta}, \qquad s > 0, \frac{1}{q} = \theta \left( \frac{1}{p} - \frac{1}{2}\right) + (1-\theta)\frac{1}{r} \]
is true, rescaling would imply 
\[ \| u_\lambda \|_{\mathcal{L}^q} \lesssim  \lambda^{-s\theta} \| u_\lambda \|_{\mathcal{L}^r}^{1-\theta} \| D^s_x u \|_{\mathcal{L}^p}^{\theta}, \qquad s > 0 \]
Taking $\lambda \to \infty$ in the last inequality would imply $\| u_\lambda \|_{\mathcal{L}^q} \to 0,$ which of course need not be true.
Therefore, the Gagliardo-Nirenberg inequality fails to hold for quasiperiodic integrals.
\end{rmk}

Finally, we would like to prove conservation laws. 
We emphasize that the quantities can be shown to be conserved via two ways: by directly taking time derivatives of averaged integrals, or by switching to Fourier series.
To illustrate our point, we prove conservation of the first three laws in Fourier space, whereas the fourth one we do directly.

Since the solutions here are taken to be smooth, so that by the equation $u_t$ is continuous and thus we are allowed to exchange time derivatives when dealing with $u(t,x)$ and $\hat{u}(t,k)$.
In what follows, we will consistently use Lemma \ref{Appendix1Id1}:
\begin{lem}\label{Appendix1Id1}
For quasiperiodic $f,g,h$, 
\[
	\sum_{k \in \ZZ^N} \FT{fg}(k) \FT{h}(-k) = \sum_{k \in \ZZ^N} \FT{f}(k) \FT{gh}(-k) = \sum_{k \in \ZZ^N} \FT{g}(k) \FT{fh}(-k).
\]
\end{lem}
For a proof see the appendix.
\begin{proof}[Proof of \eqref{CQ1}]
This is conservation of mean. Write
\[
	\frac{\partial}{\partial t} \hat{u}(k)	= \FT{uu_x}(k) + \FT{Hu_{xx}}(k).
\]
Consider the nonlinear term: 
\[
	\FT{uu_x}(k) = \frac{1}{2} \FT{\partial_x (u^2)}(k) = \frac{1}{2} \FT{\partial_x (u^2)}(k) = \frac{1}{2} \alpha \cdot k \FT{u^2}(k),
\]
so that at $k=0$ we obtain 
\[ 
	\FT{uu_x}(0) = \frac{1}{2} \alpha \cdot 0 \FT{u^2}(0) = 0.
\]
As for the linear term, note
\[ 
	\FT{Hu_{xx}}(k) = i^2 (\alpha \cdot k)^2 (-i \sgn(\alpha \cdot k)) \hat{u}(k).
\]
Clearly, this term vanishes at $k=0.$ Thus, we obtain
\[
	\frac{\partial}{\partial t} \hat{u}(0)	= 0,
\]
so that $\hat{u}(0) = \hat{u_0}(0).$
\end{proof}
\begin{proof}[Proof of \eqref{CQ2}]
Take the time derivative and use the equation:
\begin{align*}
	\frac{\partial}{\partial t}\sum_{k \in \ZZ^N} \hat{u}(k) \hat{u}(-k) &= 2 \sum_{k \in \ZZ^N} \hat{u_t}(k) \hat{u}(-k) \\
	&= 2\sum_{k \in \ZZ^N} \hat{uu_x}(k) \hat{u}(-k) + 2\sum_{k \in \ZZ^N} \hat{Hu_{xx}}(k) \hat{u}(-k).
\end{align*}
The linear term vanishes following the argument in the case of uniform estimate.
As for the nonlinear term, see Appendix (Lemma \ref{Appendix2Id1}, $n=2$) for a detailed proof that this term vanishes.
Thus, we conclude that
\[
	\frac{\partial}{\partial t}\sum_{k \in \ZZ^N} \hat{u}(k) \hat{u}(-k) = 0,
\]
so that 
\[
	\sum_{k \in \ZZ^N} \hat{u}(k) \hat{u}(-k) = \sum_{k \in \ZZ^N} \hat{u_0}(k) \hat{u_0}(-k) = \|u_0\|^2_{L^2_{qp}}.
\]
\end{proof}
\begin{proof}[Proof of \eqref{CQ3}]
Take the time derivative and use chain rule:
\begin{align*}
	\frac{\partial}{\partial t} \sum_{k \in \ZZ^N} &\frac{1}{3}\hat{u^2} (k) \hat{u}(-k) + \FT{Hu_x}(k)\hat{u}(-k) \\
	&= \sum_{k \in \ZZ^N} \frac{1}{3}[\FT{(u^2)_t} (k) \hat{u}(-k) + \hat{u^2} (k) \hat{u_t}(-k)] + \FT{H(u_t)_x}(k)\hat{u}(-k) + \FT{Hu_x}(k)\hat{u_t}(-k) \\
	&= \sum_{k \in \ZZ^N} \frac{1}{3}[2\FT{uu_t} (k) \hat{u}(-k) + \hat{u^2} (k) \hat{u_t}(-k)] + \FT{u_t}(k)\hat{Hu_x}(-k) + \FT{Hu_x}(k)\hat{u_t}(-k)
\end{align*}
Since 
\[ 
	\sum_{k \in \ZZ^N} \FT{uu_t} (k) \hat{u}(-k) = \sum_{k \in \ZZ^N}\hat{u_t} (k) \hat{u^2}(-k),
\]
we obtain that the time derivative equals
\[
	\sum_{k \in \ZZ^N} \hat{u_t}(k) \hat{u^2} (-k) + 2 \FT{u_t}(k)\FT{Hu_x}(-k).
\]
Now, use the equation:
\[
	\sum_{k \in \ZZ^N} (\FT{Hu_{xx}}(k) + \FT{uu_x}(k)) \hat{u^2} (-k) + 2 (\FT{Hu_{xx}}(k) + \FT{uu_x}(k)) \FT{Hu_x}(-k).
\]
Clearly, the term 
\[
	\sum_{k \in \ZZ^N} \FT{Hu_{xx}}(k) \FT{Hu_x}(-k) = 0,
\]
and as shown in the Appendix (Lemma \ref{Appendix2Id1}, $n=3$), 
\[
	\sum_{k \in \ZZ^N} \FT{uu_x}(k) \hat{u^2} (-k) = 0.
\]
Thus, we are left with 
\[
	\sum_{k \in \ZZ^N} \FT{Hu_{xx}}(k)\hat{u^2} (-k) + 2 \FT{uu_x}(k) \FT{Hu_x}(-k).
\]
Going back to chain rule and integration by parts yields
\[
	\sum_{k \in \ZZ^N} \FT{Hu_{xx}}(k)\hat{u^2} (-k) = -\sum_{k \in \ZZ^N} \FT{Hu_{x}}(k)\FT{(u^2)_x} (-k) = -2 \sum_{k \in \ZZ^N} \FT{Hu_{x}}(k)\FT{uu_x} (-k),
\]
so that 
\[
	\sum_{k \in \ZZ^N} \FT{Hu_{xx}}(k)\hat{u^2} (-k) + 2 \FT{uu_x}(k) \FT{Hu_x}(-k) = 0.	
\]
Thus, the term in question is conserved.
\end{proof}

\begin{proof}[Proof of \eqref{CQ4}]
We will use the following identity, called Cotlar's identity:
\[
f^2 = (Hf)^2 - 2 H(f Hf),
\]
where $f \in H^{s}_{qp}$.
Take the time derivative:
\begin{align}
	\frac{\partial}{\partial t} &\lim_{R \to \infty}\frac{1}{2R}\int^R_{-R} \frac{u^4}{4} + \frac{3}{2} u^2 Hu_x + 2u_x^2 \D x =\\
	\lim_{R \to \infty}\frac{1}{2R}\int^R_{-R} & u^3 u_t + \frac{3}{2} ( 2 uu_t Hu_x + u^2 H\partial_x u_t) + 4 u_x u_{tx} \D x.
\end{align}
Using that $H \partial_x$ is self-adjoint and integration by parts yields
\[
	\lim_{R \to \infty}\frac{1}{2R}\int^R_{-R} u^3 u_t + \frac{3}{2} ( 2 uu_t Hu_x + H\partial_x(u^2) u_t) -4 u_{xx} u_{t} \D x.
\]
This simplifies to 
\[
	\lim_{R \to \infty}\frac{1}{2R}\int^R_{-R} u^3 u_t + 3 (uu_t Hu_x + H(uu_x) u_t) -4 u_{xx} u_{t} \D x.
\]
Now, use the equation: 
\[
	\lim_{R \to \infty}\frac{1}{2R}\int^R_{-R} u^3 (uu_x + Hu_{xx}) + 3 (u Hu_x + H(uu_x))(uu_x + Hu_{xx}) -4 u_{xx} (uu_x + Hu_{xx}) \D x.
\]
For now we will assume that the three terms 
\[
	\lim_{R \to \infty}\frac{1}{2R}\int^R_{-R} u^4 u_x \D x, \quad \lim_{R \to \infty}\frac{1}{2R}\int^R_{-R} H(uu_x)(uu_x) \D x, \quad \lim_{R \to \infty}\frac{1}{2R}\int^R_{-R} H(u_{xx})u_{xx} \D x,
\]
vanish. As a result, we obtain
\[
	\lim_{R \to \infty}\frac{1}{2R}\int^R_{-R} u^3 Hu_{xx} + 3 u^2 u_x Hu_x + 3 u Hu_x Hu_{xx} + 3 Hu_{xx}H(uu_x) -4 u_{xx} uu_x \D x.
\]

Applying integration by parts on $u^3Hu_{xx}$ cancels out $3u^2 u_x Hu_{x}$, so that we are left with
\[
	\lim_{R \to \infty}\frac{1}{2R}\int^R_{-R} 3u Hu_x Hu_{xx} + 3H(uu_x)Hu_{xx} -4 u_{xx} uu_x \D x.
\]
Using that the Hilbert transform is unitary, we obtain
\[
	\lim_{R \to \infty}\frac{1}{2R}\int^R_{-R} 3u Hu_x Hu_{xx} + 3uu_x u_{xx} - 4u_{xx} uu_x \D x = \lim_{R \to \infty}\frac{1}{2R}\int^R_{-R} 3u Hu_x Hu_{xx} - u_{xx} uu_x \D x.
\]
Integration by parts yields that
\[
	\lim_{R \to \infty}\frac{1}{2R}\int^R_{-R} -\frac{3}{2} u_x (Hu_x)^2 + \frac{1}{2} u_x (u_x)^2 \D x = \lim_{R \to \infty}\frac{1}{2R}\int^R_{-R} (-\frac{3}{2} (Hu_x)^2 + \frac{1}{2}(u_x)^2 ) u_x \D x
\]
Now, use Cotlar's identity with $f = u_x$ to rewrite the integrand as 
\[ 
	-\frac{3}{2} (Hu_x)^2 + \frac{1}{2}(Hu_x)^2 - H(u_x Hu_x) = - (Hu_x)^2 - H(u_x Hu_x),
\]
to obtain 
\[
\lim_{R \to \infty}\frac{1}{2R}\int^R_{-R} (- (Hu_x)^2 - H(u_x Hu_x))u_x \D x.
\]
Finally, using the anti-self-adjointness of Hilbert transform, 
\[
	\lim_{R \to \infty}\frac{1}{2R}\int^R_{-R} H(u_x Hu_x))u_x \D x = -\lim_{R \to \infty}\frac{1}{2R}\int^R_{-R} u_x Hu_x Hu_x \D x = - \lim_{R \to \infty}\frac{1}{2R}\int^R_{-R} u_x (Hu_x)^2 \D x.
\]
Combining, we obtain
\[
	\lim_{R \to \infty}\frac{1}{2R}\int^R_{-R} (-(Hu_x)^2 -  H(u_x Hu_x))u_x \D x = \lim_{R \to \infty}\frac{1}{2R}\int^R_{-R} (-(Hu_x)^2 + (Hu_x)^2)u_x \D x = 0,
\]
as desired.
\end{proof}

\appendix 
\section{Various results}
\begin{lem}
For quasiperiodic $f,g,h$, 
\[
	\sum_{k \in \ZZ^N} \FT{fg}(k) \FT{h}(-k) = \sum_{k \in \ZZ^N} \FT{f}(k) \FT{gh}(-k) = \sum_{k \in \ZZ^N} \FT{g}(k) \FT{fh}(-k).
\]
\end{lem}
\begin{proof}
Writing a product as a convolution, we obtain
\begin{align*}
	\sum_{k \in \ZZ^N} \FT{fg}(k) \FT{h}(-k) &= \sum_{k \in \ZZ^N} \sum_{j \in \ZZ^N} \FT{f}(j) \FT{g}(k-j) \FT{h}(-k) \\
	&= \sum_{j \in \ZZ^N} \FT{f}(j) \sum_{k \in \ZZ^N} \FT{g}(k-j) \FT{h}(-k) \\
	&= \sum_{j \in \ZZ^N} \FT{f}(j) \FT{gh}(-j) \\
	&=\sum_{k \in \ZZ^N} \FT{f}(k) \FT{gh}(-k).
\end{align*}
On the third line, we interchange the summation and recognize a convolution in $k$. 
Note that at the end, we have relabeled the index of summation.
The second identity follows similarly from using 
\[
	\FT{fg}(k) = \sum_{j\in\ZZ^N} \FT{f}(k-j) \FT{g}(j).
\]
Thus we obtain the identities in question.
\end{proof}

\begin{lem}\label{Appendix2Id1} For any natural $n \geq 0$,
we have:
\[ \sum_{k \in \ZZ^N} \widehat{u^n}(k) \widehat{u_x}(-k) = 0.\]
\end{lem}
\begin{proof}
Rewriting the sum as a convolution and using chain rule yields
\[ \sum_{k \in \ZZ^N} \widehat{u^n}(k) \widehat{u_x}(-k) = \FT{(u^n u_x)}(0) = \frac{1}{n+1}\FT{(u^{n+1})_x}(0) = \frac{1}{n+1} \Big( i\alpha \cdot k \FT{u^{n+1}}(k)\Big)\Big|_{k=0} = 0,
\]
so that we have the result.
\end{proof}
	
\begin{lem}\label{IntegrationByParts}
We have:
\[\sum_{k\in\ZZ^N} \widehat{v\partial_x u}(k)\hat{u}(-k)
= -\frac{1}{2} \sum_{k\in\ZZ^N} \widehat{v_x u}(k) \hat{u}(-k).
\]
\end{lem}
\begin{proof}
We use Lemma \ref{Appendix1Id1} to move $\partial_x u$ into the $\hat{u}(-k)$ term:
\begin{align*}
	\sum_{k\in\ZZ^N} \widehat{v \partial_x u}(k)\hat{u}(-k) &= \sum_{k\in\ZZ^N} \widehat{v}(k)\widehat{u \partial_x u}(-k).
\end{align*}
Recognizing the derivative term $u \partial_x u = \frac{1}{2} \partial_x(u^2)$, we then write 
\begin{align*}
	\sum_{k\in\ZZ^N} \widehat{v}(k)\widehat{u \partial_x u}(-k) &= \frac{1}{2} \sum_{k\in\ZZ^N} \widehat{v}(k)\widehat{\partial_x(u^2)}(-k) \\
	&= \frac{1}{2} \sum_{k\in\ZZ^N} (\alpha \cdot (-k))\widehat{v}(k)\widehat{u^2}(-k) \\
	&= - \frac{1}{2} \sum_{k\in\ZZ^N} (\alpha \cdot k)\widehat{v}(k)\widehat{u^2}(-k),
\end{align*}
where we rewrote the derivative in Fourier variables and pulled the minus sign out of the sum. 
The final term can be written as 
\[- \frac{1}{2} \sum_{k\in\ZZ^N}\widehat{v_x}(k)\widehat{u^2}(-k),
\]
and using Lemma \ref{Appendix1Id1}, we can shift one of $u$ inside the $\widehat{u^2}$ term into the $\widehat{v}$ term. 
We then obtain the desired identity.
\end{proof}

\begin{lem}[Cotlar's identity]\label{AppendixCotlarsId}
Let $f \in H^{s}_{qp}$. Then, 
\[
(Hf)^2 - f^2 = 2 H(f Hf).
\]
\end{lem}
\begin{proof}
This proof is a straightforward extension of a standard argument to the quasiperiodic setting.

Since $f \in H^s_{qp}$, we know that the terms $(Hf)^2$ and $H(f Hf)$ exist due to the algebra property. 
In particular, all the terms involved have their Fourier series expressions, so it is enough to prove the identity in Fourier space.

In Fourier space, the identity is given by
\begin{align*}
	\FT{(Hf)^2}(k) - \FT{f^2}(k) &= 2 \FT{H(fHf)} \\
	&= 2 (-i)\sgn (\alpha \cdot k) \FT{(f Hf)}(k) \\
	&= (-i)\sgn (\alpha \cdot k) \FT{(f Hf)}(k) + (-i) \sgn (\alpha \cdot k) \FT{(f Hf)}(k).
\end{align*} 
Using convolution theorem we obtain
\begin{align*}
	\sum_{j \in \ZZ^N} &\FT{Hf}(j) \FT{Hf}(k-j) - \hat{f}(j) \hat{f}(k-j) \\
	&= (-i)\sgn (\alpha \cdot k) \sum_{j \in \ZZ^N} \hat{f}(j) \FT{Hf}(k-j) + (-i)\sgn (\alpha \cdot k) \sum_{j \in \ZZ^N} \FT{Hf}(j) \hat{f}(k-j) 
\end{align*}
Using the definition of Hilbert transform results in
\begin{align*}
	\sum_{j \in \ZZ^N} &((-i)^2\sgn(\alpha \cdot j) \sgn(\alpha \cdot k-j)  - 1 )\hat{f}(j) \hat{f}(k-j) \\
	&= (-i)\sgn (\alpha \cdot k) \sum_{j \in \ZZ^N} (-i) \sgn(\alpha \cdot k-j) \hat{f}(j) \hat{f}(k-j) + (-i)\sgn (\alpha \cdot k) \sum_{j \in \ZZ^N} (-i) \sgn(\alpha \cdot j) \hat{f}(j) \hat{f}(k-j) \\
	&=  \sum_{j \in \ZZ^N} (-i)^2 ( \sgn (\alpha \cdot k) \sgn(\alpha \cdot k-j) +  \sgn (\alpha \cdot k) \sgn(\alpha \cdot j))\hat{f}(j) \hat{f}(k-j),
\end{align*}
so we just need prove that 
\[ (-i)^2\sgn(\alpha \cdot j) \sgn(\alpha \cdot k-j)  - 1  = (-i)^2 ( \sgn (\alpha \cdot k) \sgn(\alpha \cdot k-j) +  \sgn (\alpha \cdot k) \sgn(\alpha \cdot j)), \]
in other words, 
\[ -\sgn(\alpha \cdot j) \sgn(\alpha \cdot k-j)  - 1  = -( \sgn (\alpha \cdot k) \sgn(\alpha \cdot k-j) +  \sgn (\alpha \cdot k) \sgn(\alpha \cdot j)), \]
which is the same as 
\begin{equation}\label{cotlareq}
	\sgn(\alpha \cdot k-j) [(\sgn(\alpha \cdot k)) - \sgn(\alpha \cdot j)] = 1 - \sgn (\alpha \cdot j) \sgn(\alpha \cdot k).
\end{equation}
When both $\alpha \cdot k$ and $\alpha \cdot j$ are of the same sign, both sides of \eqref{cotlareq} evaluate to 0.
When $\alpha \cdot k$ and $\alpha \cdot j$ are of opposite signs, both sides of \eqref{cotlareq} evaluate to 2.
Thus, \eqref{cotlareq} holds.
\end{proof}

\section*{Acknowledgments} 
SA gratefully acknowledges useful conversations with Thierry Laurens and Jason Chen. 
DMA gratefully acknowledges support from the National Science Foundation through grant DMS-2307638.

\bibliographystyle{amsplain}
{\small\bibliography{references}}

\end{document}